\documentclass[10pt,letterpapper]{amsart}

\usepackage{graphicx,hyperref}
\usepackage[hmargin=1in, vmargin=1in]{geometry}
\renewcommand{\theequation}{\thesection.\arabic{equation}}

\theoremstyle{plain}
\newtheorem*{main}{Main Theorem}
\newtheorem{thm}{Theorem}[section]
\newtheorem{lem}{Lemma}[section]
\newtheorem{prop}[lem]{Proposition}

\theoremstyle{definition}
\newtheorem{defn}{Definition}
\newtheorem{rmk}{Remark}

\renewcommand{\[}{\begin{equation}\notag\begin{aligned}}
\renewcommand{\]}{\end{aligned}\end{equation}}
\newcommand{\beq}[1]{\begin{equation}\label{#1}\begin{aligned}}
\newcommand{\beqeps}[1]%
  {\addtocounter{equation}{1}\begin{equation}\label{#1}\tag{$\theequation\epsilon$}\begin{aligned}}

\newcommand{\p}{\begin{pmatrix}}
\newcommand{\pp}{\end{pmatrix}}

\newcommand{\din}{\mathrm{in}}
\newcommand{\dout}{\mathrm{out}}

\newcommand{\myflow}[1]{\underset{\eqref{#1}}{\bullet}}

\author{Ting-Hao Hsu}
\address{Department of Mathematics\\ The Ohio State University\\ Columbus, OH 43210}
\email{hsu.296@osu.edu}
\title[Viscous singular shock profiles for two-phase flow]%
{Viscous singular shock profiles for a system of\\
conservation laws modeling two-phase flow}


\begin{document}

\maketitle
\begin{center}
\today
\end{center}
\begin{abstract}
This paper is concerned with singular shocks
for a system of conservation laws modeling incompressible two-phase fluid flow.
We prove the existence of viscous profiles
using the geometric singular perturbation theory.
Weak convergence and growth rates of the unbounded family of solutions are also obtained.
\end{abstract}

\section{Introduction} \label{sec_intro}
Keyfitz et al \cite{Keyfitz:2003,Keyfitz:2004}
considered the system of conservation laws \beq{claw_bv}
  &\beta_t+(vB_1(\beta))_x=0\\
  &v_t+(v^2B_2(\beta))_x=0
\] where $t\ge0$, $x\in\mathbb R$, $v\in\mathbb R$, $\beta\in [\rho_1,\rho_2]$
with $\rho_2<\rho_1$ and \beq{def_B}
  B_1(\beta)= \frac{(\beta-\rho_1)(\beta-\rho_2)}\beta,\quad
  B_2(\beta)= \frac{\beta^2-\rho_1\rho_2}{2\beta^2}.
\]
For Riemann problems with data in feasible regions,
they constructed uniquely defined admissible solutions.
It can be readily shown that this system is not everywhere hyperbolic,
and hence standard methods does not apply
(see e.g. \cite{Smoller:1983,Dafermos:2010}).
To resolve this problem,
along with rarefaction waves and regular shocks,
the concept of singular shocks was adopted.
A singular shock solution,
roughly speaking,
is a distribution which contains delta measures
and is the weak limit of a sequence of approximate viscous solutions.
For details of the definition, we refer to \cite{Sever:2007, Keyfitz:2011}.

The existence of singular shocks for \eqref{claw_bv}
was proved in \cite{Keyfitz:2004}.
In that work, 
for certain Riemann data \beq{ic_riemann_bv}
  (\beta,v)(x,0)
  = \begin{cases}
    (\beta_{L},v_{L}),&x<0\\
    (\beta_{R},v_{R}),&x>0
  \end{cases}
\]
approximate solutions of
the regularized system via Dafermos regularization \beqeps{dafermos_bv}
  &\beta_t+(vB_1(\beta))_x=\epsilon t\beta_{xx}\\
  &v_t+(v^2B_2(\beta))_x=\epsilon t v_{xx}
\]
were constructed.
A family of exact solutions of \eqref{dafermos_bv} and \eqref{ic_riemann_bv},
rather than approximate solutions,
is called a \emph{viscous profile}.
In this paper, we prove existence of viscous profile,
also we give descriptions of their limiting behavior
including weak convergence and growth rates.
The main tool in our study is the \emph{Geometric Singular Perturbation Theory} (GSPT),
which will be introduced in later sections.
The use of this tool on singular shocks was
first introduced in the pioneering work of Schecter \cite{Schecter:2004}.

The system \eqref{claw_bv} is equivalent to
a two-fluid model for incompressible two-phase flow \cite[p.248]{Drew:1999}
of the form \beq{claw_alpha}
  &\partial_t(\alpha_i)+\partial_x(\alpha_iu_i)=0\\
  &\partial_t(\alpha_i\rho_iu_i)+\partial_x(\alpha_i\rho_iu_i^2)+\alpha_i\partial_xp_i= F_i,
  \quad i=1,2,
\] where the drag terms $F_i$ are neglected
and the pressure terms satisfy $p_1=p_2$.
To reduce \eqref{claw_alpha} to \eqref{claw_bv},
in \cite{Keyfitz:2003} the volume fractions $\alpha_1$ and $\alpha_2=1-\alpha_1$
have been replaced by
a density-weighted volume element
$\beta=\rho_2\alpha_1+\rho_1\alpha_2$
and the momentum equations replaced by a single equation for
the momentum difference
$v=\rho_1u_1-\rho_2u_2-(\rho_1-\rho_2)K$,
where $K=\alpha_2u_1+\alpha_2u_2$ is taken to be zero.
This is a simple example of continuous model for two-phase flow,
but it shares with other continuous models the property of
changing type -- that is, it is not hyperbolic for some (in this case, most) states.

The purpose of this study is to shed light on the mathematical properties
of the change-of-type system
that appear in continuous models of two-phase flow.
The original studies \cite{Keyfitz:2003,Keyfitz:2004}
showed the existence of self-similar solutions
with reasonable properties.
Specifically, the singular shocks that appear can be considered to be propagating phase boundaries.
In this paper,
we focus on viscous profiles of singular shocks
and unveil some of their limiting behavior.

In Section \ref{sec_main},
we state our main result,
and in Section \ref{sec_assumptions}
the validity of the assumptions of the theorem is discussed,
with some proofs for the sufficient conditions postponed to Section \ref{sec_h3}.
In Section \ref{sec_GSPT},
we recall and enhance some tools in GSPT, 
including Fenichel's Theorems and the Exchange Lemma.
Section \ref{sec_singular_config}
is devoted to describing the structure of the system.
The proof of the main theorem is completed in Section \ref{sec_complete},
and numerical simulations are shown in Section \ref{sec_simulation}.

\section{Main Result}\label{sec_main}
In standard notation for conservation laws,
we write \eqref{claw_bv} as
\beq{claw_u}
  u_t+f(u)_x=0,
\] where $u=(\beta,v)$,
and write Riemann data for Riemann problems
in the form \beq{bc_riemann}
  u(x,0)= u_L+ (u_R-u_L) \mathrm{H}(x),
\]
where $\mathrm{H}(x)$ is the step function taking value $0$ if $x<0$; $1$ if $x>0$.

We study the systems that approximate \eqref{claw_u}
via the Dafermos regularization: \beqeps{dafermos_u}
  u_t+f(u)_x=\epsilon t u_{xx}
\] for small $\epsilon> 0$.
Using the self-similar variable $\xi = x/t$,
the system is converted to \beqeps{dafermos_similar}
  -\xi \frac{d}{d\xi} u+ \frac{d}{d\xi}\big(f(u)\big)
  = \epsilon \frac{d^2}{d\xi^2}u,
\] and the initial condition \eqref{bc_riemann} becomes \beq{bc_u_infty}
  u(-\infty)=u_L,\; u(+\infty)=u_R.
\]
The system \eqref{dafermos_similar} is equivalent to \beqeps{dafermos_similar_xi}
  &-\epsilon u_\xi= f(u)- \xi u- w\\
  &w_\xi= - u\\
\] or, up to a rescaling of time, \beqeps{sf_u}
  &\dot u= f(u)- \xi u- w\\
  &\dot{w}= -\epsilon u\\
  &\dot\xi= \epsilon.
\] The time variable in \eqref{sf_u}
is implicitly defined by the equation of $\dot\xi$.
When $\epsilon=0$, \eqref{sf_u} is reduced to \beq{fast_u}
  &\dot u= f(u)-\xi u- w\\
  &\dot{w}= 0,\;
  \dot\xi= 0.
\]
Returning to the $(\beta,v)$ notation,
the system \eqref{sf_u} is written as \beqeps{sf_bv}
  &\dot\beta= -B_1(\beta)v- \xi\beta- w_1\\
  &\dot v= -B_2(\beta)v^2- \xi v- w_2\\
  &\dot{w}_1= -\epsilon\beta\\
  &\dot{w}_2= -\epsilon v\\
  &\dot\xi= \epsilon,
\] and \eqref{fast_u} becomes \beq{fast_bv}
  &\dot\beta= -B_1(\beta)v- \xi\beta- w_1\\
  &\dot v= -B_2(\beta)v^2- \xi v- w_2\\
  &\dot{w}_1= 0,\;
  \dot{w}_2= 0,\;
  \dot\xi= 0.
\]
The linearization
at any equilibrium $(\beta,v,w_1,w_2,\xi)$ for \eqref{fast_bv}
has eigenvalues $\lambda_{\pm}(\beta,v)-\xi$,
where \beq{def_lambda}
  \lambda_{\pm}(u)
  = 2vB_2(\beta)\pm v\sqrt{B_1(\beta)B_2'(\beta)}.
\] Note that $\mathrm{Re}(\lambda_\pm(u))=2vB_2(\beta)$
since $B_1(\beta)B_2'(\beta)\le 0$
when $\rho_2\le\beta\le\rho_1$.
Moreover,
the system is nonhyperbolic everywhere in the physical region
except on the union of the lines $\{\beta=\rho_1\}$, $\{\beta=\rho_2\}$, and $\{v=0\}$.

An \emph{over-compressive shock region}
is a region where the condition $\mathrm{(H1)}$
defined below holds.
It was shown in \cite{Keyfitz:2004}
that any data in an over-compressive shock region
admits a singular shock solution,
and the shock speed $s$ is defined by \eqref{def_s} below.
Our main theorem 
confirms Dafermos profiles
in a subset of this region.

\begin{main}
Consider the Riemann problem \eqref{claw_bv}, \eqref{ic_riemann_bv}.
Let $u_L=(\beta_L,v_L)$ and $u_R=(\beta_R,v_R)$
be two points in $[\rho_1,\rho_2]\times (0,\infty)$ with $\beta_R\ne \beta_L$,
and let \begin{align}
  &s=\frac{v_LB_1(\beta_L)-v_RB_1(\beta_R)}{\beta_L-\beta_R}\label{def_s}\\
  &w_L=f(u_L)-su_L,\;
  w_R=f(u_R)-su_R\label{def_w}\\
  &e_0=w_{2L}-w_{2R}\label{def_deficit}
\end{align} where we denote $w_L=(w_{1L},w_{2L})$ and $w_R=(w_{1R},w_{2R})$.
Assume \begin{enumerate}
  \item[$\mathrm{(H1)}$] $\mathrm{Re}(\lambda_\pm(u_R))<s <\mathrm{Re}(\lambda_\pm(u_L))$,
  where $\lambda_\pm(u)$ are defined in \eqref{def_lambda}.
  \item[$\mathrm{(H2)}$] $e_0>0$.
  \item[$\mathrm{(H3)}$] For the system \eqref{fast_bv},
  there exists a trajectory
  joining $(\beta_L,v_L,w_{L},s)$
  and  $(\rho_1,+\infty,w_{L},s)$,
  and a trajectory joining $(\beta_R,v_R,w_{R},s)$
  and  $(\rho_2,+\infty,w_{R},s)$.
\end{enumerate}
Then there is a singular shock with Dafermos profile
for the Riemann data $(u_L,u_R)$.
That is, for each small $\epsilon>0$,
there is a solution $\tilde{u}^\epsilon(\xi)$ of \eqref{dafermos_similar} and \eqref{bc_u_infty},
and $\tilde{u}_\epsilon(\xi)$ becomes unbounded as $\epsilon\to 0$.
Indeed, \beq{est_v_max}
  \max_\xi \Big(\epsilon\log \tilde{v}_\epsilon(\xi)\Big)
  =  \frac{(\rho_1-\rho_2)\,(w_{2L}-w_{2R})}{\rho_1+\rho_2}+ o(1)
  \quad\text{as }\epsilon\to 0.
\] Moreover,
if we set $u_\epsilon(x,t)=\tilde{u}_\epsilon(x/t)$,
then $u_\epsilon(x,t)$ is a solution of \eqref{dafermos_u}
and \begin{subequations}\label{weak_limit}
  \begin{align}
  &\beta_\epsilon\rightharpoonup \beta_L+ (\beta_R-\beta_L) \mathrm{H}(x-st)
    \label{weak_limit_beta}\\
    &v_\epsilon\rightharpoonup v_L+ (v_R-v_L) \mathrm{H}(x-st)
    + \frac{e_0}{\sqrt{1+s^2}}t\delta_{\{x=st\}}
    \label{weak_limit_v}
  \end{align}
\end{subequations}
in the sense of distributions as $\epsilon\to 0$.
\end{main}

The trajectories in $\mathrm{(H3)}$ are illustrated in Fig \ref{fig_H3}.

\begin{figure}[t]\centering
\includegraphics[trim = 1cm 6.8cm 2cm 6cm, clip, width=.49\textwidth]{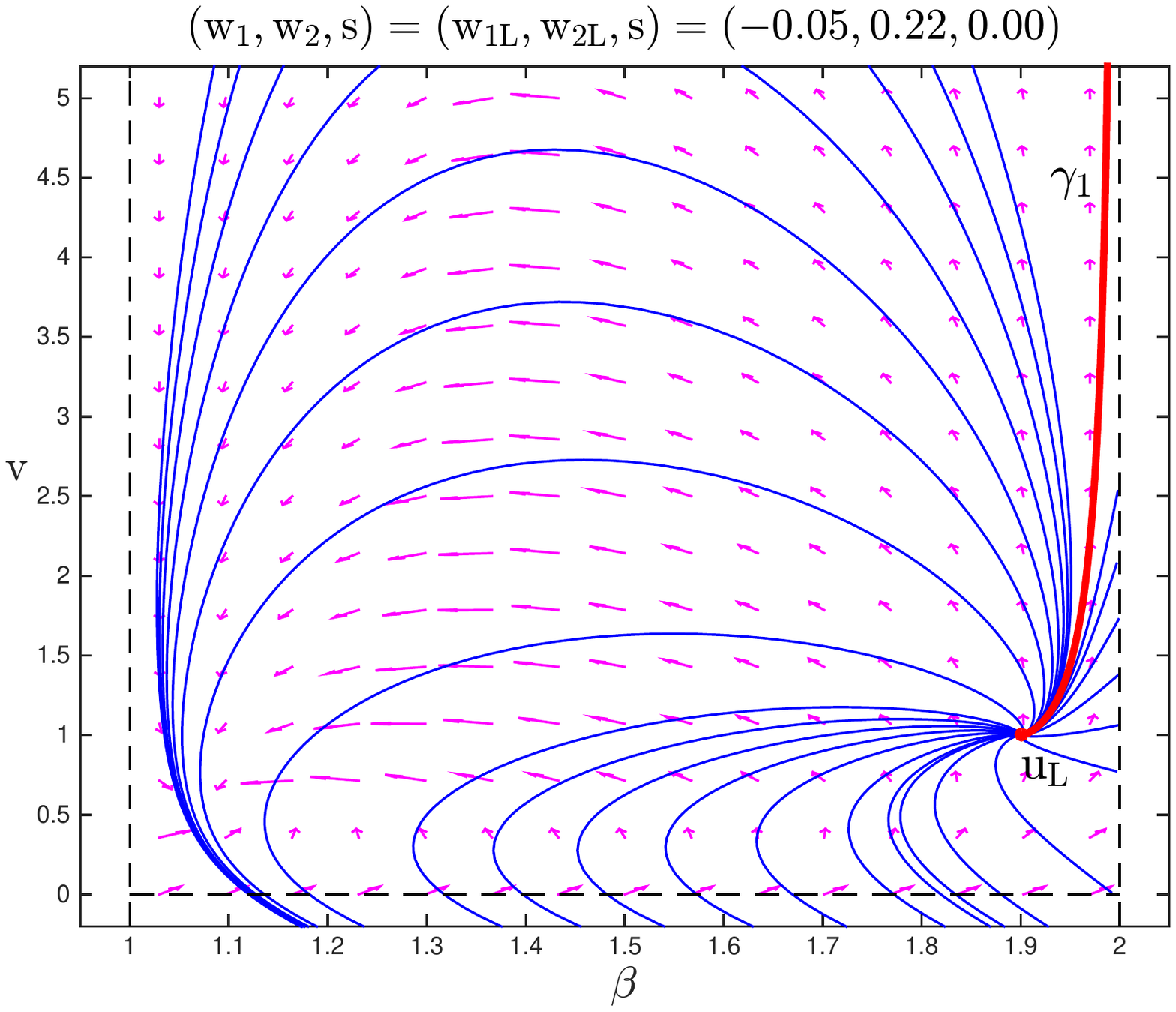}
\includegraphics[trim = 1cm 6.8cm 2cm 6cm, clip, width=.49\textwidth]{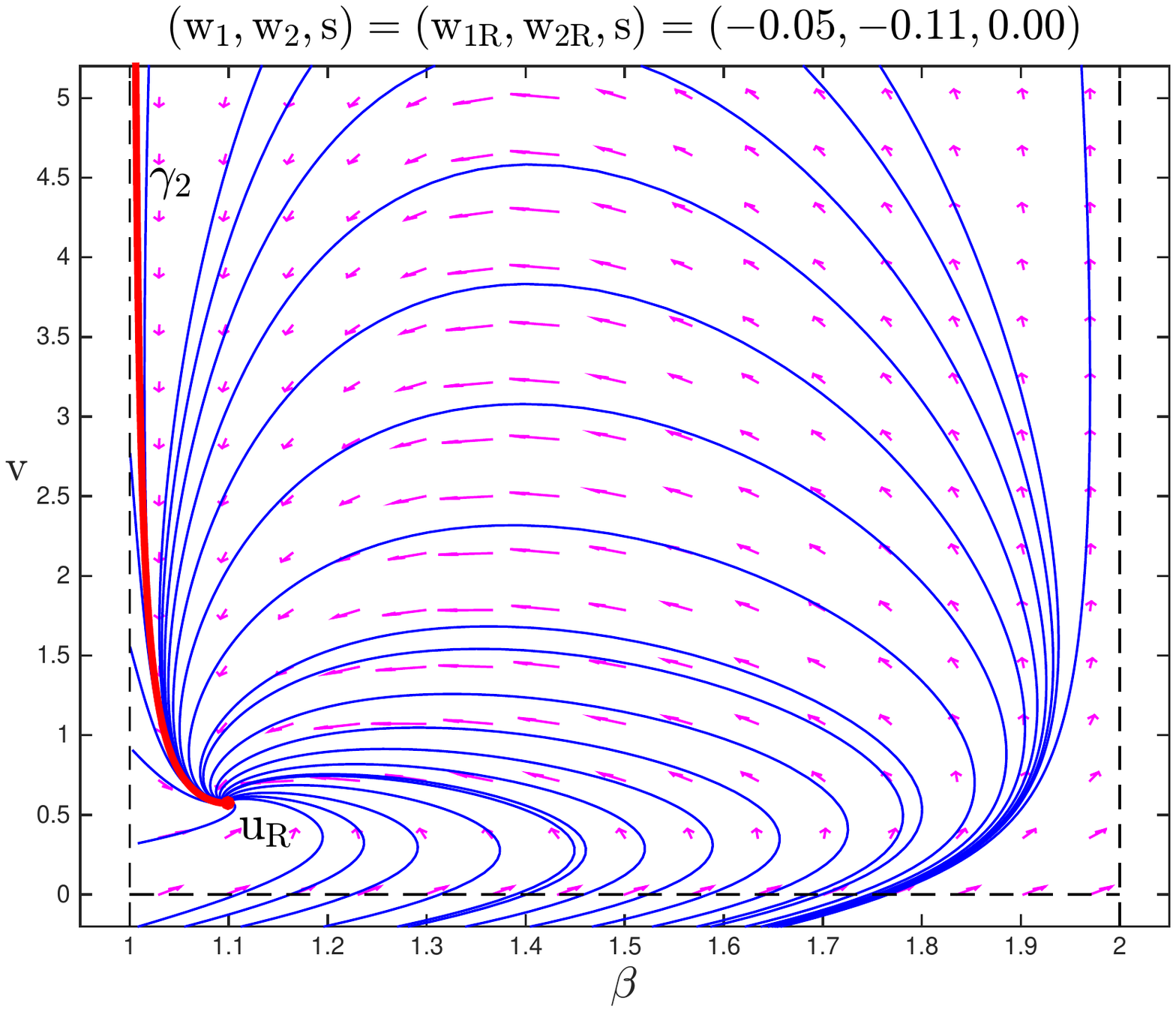}
\caption{Phase portraits for $\dot{u}=f(u)-su-w$ with fixed $s$ and $w$.
The singular trajectories in $\mathrm{(H3)}$
are denoted by $\gamma_1$ and $\gamma_2$.}
\label{fig_H3}
\end{figure}

\begin{rmk}
A similar result holds if $v_L<0$ and $v_R<0$.
In that case,
the assumption $e_0>0$ in $\mathrm{(H2)}$ is replaced by $e_0<0$,
and $+\infty$ in $\mathrm{(H3)}$ is replaced by $-\infty$.
\end{rmk}

The notation $t\delta_{\{x=st\}}$
in \eqref{weak_limit_v} denotes,
following \cite{Tan:1994,Chen:2003},
the functional on $C^\infty_c(\mathbb R\times\mathbb R^+)$
defined by \beq{def_delta}
  \langle t\delta_{\{x=st\}},\varphi\rangle
  = \int_0^\infty t\varphi(st,t)\sqrt{1+s^2}\;dt.
\] The weight $\sqrt{1+s^2}$ in the integral
is to normalize the functional
so that it is independent of parametrization of the line $\{x=st\}$.

The estimate \eqref{est_v_max} confirms the asymptotic behavior
conjectured in \cite{Keyfitz:2003}.
In \cite{Keyfitz:2004},
some approximate solutions for the Dafermos regularization were constructed,
but they were not exact solutions to \eqref{dafermos_similar}.
The results in the main theorem
can also be compared to \cite{Schecter:2004} and \cite{Keyfitz:2012},
where Dafermos profiles were constructed for a system motivated by gas dynamics.
Those authors obtained families of unbounded solutions to \eqref{dafermos_similar},
but they did not give descriptions of asymptotic behaviors the of solutions.

The assumption $\mathrm{(H3)}$
says that there exist solutions of \eqref{fast_bv} of the form \beq{def_gamma}
  \gamma_1=\big(\beta_1(\xi),v_1(\xi),w_{1L},w_{2L},s\big),\;
  \gamma_2=\big(\beta_2(\xi),v_2(\xi),w_{1R},w_{2R},s\big)
\] satisfying \beq{limit_gamma1}
  \lim_{\xi\to-\infty}(\beta_1(\xi),v_1(\xi))= (\beta_L,v_L),\quad
  \lim_{\xi\to\infty}(\beta_1(\xi),v_1(\xi))= (\rho_1,+\infty)
\] and \beq{limit_gamma2}
  \lim_{\xi\to-\infty}(\beta_2(\xi),v_2(\xi))= (\rho_2,+\infty),\quad
  \lim_{\xi\to\infty}(\beta_2(\xi),v_2(\xi))= (\beta_R,v_R).
\] A local analysis for \eqref{fast_bv}
with $(w,\xi)=(w_{L},s)$ and $(w_{R},s)$, respectively,
at $(\rho_1,+\infty)$ and $(\rho_2,+\infty)$
shows that the trajectories in $\mathrm{(H3)}$, if they exist, are unique.

A sample set of data for which ($\mathrm{H}$1)-($\mathrm{H}$3) holds is,
following \cite{Keyfitz:2003}, \beq{sample_data}
  \rho_1=2,\;
  \rho_2=1,\;
  u_L=(1.9,1.0),\;
  u_R=(1.1,1.1/1.9).
\] This will be verified in the next subsection.

\section{Sufficient Conditions for $\mathrm{(H1)}$-$\mathrm{(H3)}$} \label{sec_assumptions}
The regions at which $\mathrm{(H1)}$ holds,
or the over-compressive shock regions,
can be described by the following

\begin{prop}
\label{prop_overcompressive}
In the Riemann problem \eqref{claw_bv}, \eqref{bc_riemann},
let $u_L=(\beta_L,v_L)$ and $u_R=(\beta_R,v_R)$
be two points in $[\rho_1,\rho_2]\times (0,\infty)$.
Then $\mathrm{(H1)}$ holds if and only if $u_R$
lies in the interior of a cusped triangular region
bounded by the curves \beq{curve_s_lambda_m}
  v= v_L\left(\frac{B_1(\beta_L)-2B_2(\beta_L)(\beta_L-\beta)}{B_1(\beta)}\right),\;
  \rho_2<\beta_R<\beta_L,
\]
and \beq{curve_s_lambda_p}
  v= v_L\left(\frac{B_1(\beta_L)}{B_1(\beta)+2B_2(\beta)(\beta_L-\beta)}\right),\;
  \rho_2<\beta_R<\beta_L.
\]
On the boundary segment \eqref{curve_s_lambda_m}, $s=\mathrm{Re}(\lambda_{\pm}(u_L))$,
and on  \eqref{curve_s_lambda_p}, $s=\mathrm{Re}(\lambda_{\pm}(u_R))$.
\end{prop}

The curves defined by \eqref{curve_s_lambda_m}
and \eqref{curve_s_lambda_p}
and the region where over-compressive shock solution exist
are illustrated in Fig \ref{fig_oc}.

\begin{proof}
This follows from a direct calculation.
See \cite[Corollary 3.1]{Keyfitz:2003}.
\end{proof}

\begin{figure}[t]\centering
{
\includegraphics[trim = 2cm 6.8cm 2cm 7cm, clip, width=.45\textwidth]{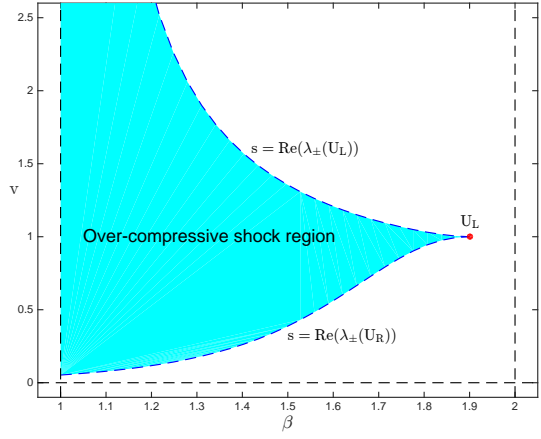}
}
\caption{The over-compressive shock region for $\rho_1=2$, $\rho_2=1$, $U_L=(1.1,1.1/1.9)$.}
\label{fig_oc}
\end{figure}

The following proposition asserts that $\mathrm{(H2)}$ is implied by $\mathrm{(H1)}$.

\begin{prop}
\label{prop_positive_deficit}
In the Riemann problem \eqref{claw_bv}, \eqref{bc_riemann},
if the Riemann data lie in an over-compressive shock region
in $[\rho_1,\rho_2]\times (0,\infty)$,
then $\mathrm{(H2)}$ holds.
\end{prop}

\begin{proof}
See \cite[Section 3.1]{Keyfitz:2003}.
\end{proof}

The assumption $\mathrm{(H3)}$
is a condition on dynamics of $2$-dimensional systems.
Analyzing phase portraits we have the following

\begin{prop}\label{prop_h3}
Given Riemann data in an over-compressive shock region
in $[\rho_1,\rho_2]\times (0,\infty)$,
if $\beta_R<\sqrt{\rho_1\rho_2}<\beta_L$,
$w_{1L}<0$, $w_{2R}<0<w_{2L}$,
and $|s|$ is sufficiently small,
then $\mathrm{(H3)}$ holds.
\end{prop}

\begin{proof}
See Section \ref{sec_h3}.
\end{proof}

Proposition \ref{prop_positive_deficit} says that
$\mathrm{(H2)}$ holds whenever $\mathrm{(H1)}$ holds,
so the Main Theorem requires only $\mathrm{(H1)}$ and $\mathrm{(H3)}$.
The author believes
that $\mathrm{(H3)}$ is also a consequence of $\mathrm{(H1)}$.
This needs further work to be verified.

For the sample set of data \eqref{sample_data},
we have \beq{sample_w}
  (w_{1L},w_{2L})=(-.05,.22),\;
  (w_{1R},w_{2R})=(-.05,-.11),\;
  s=0.
\] Since $w_{2R}<w_{2L}$, $\mathrm{(H2)}$ holds.
From Proposition \ref{prop_overcompressive} and \ref{prop_h3},
$\mathrm{(H1)}$ and $\mathrm{(H3)}$ also hold.
Hence the main theorem applies.
Note that the conditions $\mathrm{(H1)}$-$\mathrm{(H3)}$
persist under perturbation of
the Riemann data $(u_L,u_R)$,
so those assumptions still hold
for any data close to \eqref{sample_data}.

\section{Geometric Singular Perturbation Theory}
\label{sec_GSPT}
Our main goal is to solve
the boundary value problem \eqref{dafermos_similar} and \eqref{bc_u_infty}.
Note that \eqref{dafermos_similar} is a \emph{singularly perturbed equation}
since the perturbation $\epsilon\frac{d^2}{d\xi^2}u$
has a higher order derivative than the other terms in the equation.
To deal with singularly perturbed equations,
we will apply \emph{Geometric Singular Perturbation Theory} (GSPT).
The idea of GSPT 
is to first study a set of subsystems which forms a decomposition of a system,
and then to use the information for the subsystems
to conclude results for the original system.
Prototypical examples include
relaxation oscillations for
forced Van der Pol Equations \cite{Dumortier:1996,Krupa:2001,Krupa:2001a}
and FitzHugh-Nagumo Equations \cite{Jones:1991,Keyfitz:2003,Liu:2006}.
Surveys on this topic can be found in
\cite{Jones:1995,Kaper:1999,Kaper:2001,Rubin:2002}.

In Section \ref{subsec_fenichel} and \ref{subsec_silnikov},
we recall some fundamental theorems in GSPT. 
In Section \ref{subsec_EL}
we state and give new proofs for a version
of the Exchange Lemma.

\subsection{Fenichel's Theory for Fast-Slow Systems} \label{subsec_fenichel}
Note that \eqref{sf_u} is a \emph{fast-slow system},
which means that the system is of the form \beqeps{sf_xy}
  &\dot x=f(x,y,\epsilon)\\
  &\dot y=\epsilon g(x,y,\epsilon)
\] where $(x,y)\in\mathbb R^n\times \mathbb R^l$, 
and $\epsilon$ is a parameter.
In order to deal with fast-slow systems,
Fenichel's Theory
was developed in \cite{Fenichel:1973,Fenichel:1977,Fenichel:1979}.
Some expositions for that theory can be found in \cite{Wiggins:1994,Jones:1995}.

An important feature of a fast-slow system
is that the system can be decomposed into two subsystems:
the \emph{limiting fast system}
and the \emph{limiting slow system}.
The limiting fast system is obtained by taking $\epsilon=0$ in \eqref{sf_xy};
that is, \beq{fast_xy}
  &\dot x=f(x,y,0)\\
  &\dot y=0.
\] 
On the other hand, note that the system \eqref{sf_xy} can be converted to,
after a rescaling of time,
\beqeps{sf_xy_singular}
  &\epsilon x'= f(x,y,\epsilon)\\
  &y'= g(x,y,\epsilon).
\] Taking $\epsilon=0$ in \eqref{sf_xy_singular}, 
we obtain the limiting slow system \beq{slow_xy}
  &0=f(x,y,0)\\
  &y'=g(x,y,0).
\] 
Note that the limiting slow system \eqref{slow_xy}
describes dynamics on the set of critical points of 
the limiting fast system \eqref{fast_xy},
so we will need to piece together the information of
the limiting fast system and the limiting slow system
in the vicinity of the set of critical points.
To piece this information together,
\emph{normal hyperbolicity} defined below will be a crucial condition.
\begin{defn}\label{defn_normally_hyperbolic}
A \emph{critical manifold} $\mathcal S_0$ for \eqref{fast_xy}
is an $l$-dimensional manifold consisting of critical points of \eqref{fast_xy}.
A critical manifold is \emph{normally hyperbolic}
if $D_xf(x,y,0)|_{\mathcal S_0}$ is hyperbolic.
That is,
at any point $(x_0,y_0)\in \mathcal S_0$,
all eigenvalues of $D_xf(x,y,0)|_{(x_0,y_0)}$
have nonzero real part.
\end{defn}

Now we turn to discussing normal hyperbolicity
for general systems \beq{deq_x}
  \dot{z}= F(z),
\] where $z\in \mathbb R^N$, $N\ge 1$.
A manifold $\mathcal S\subset \mathbb R^N$
is \emph{locally invariant}
if for any point $p\in \mathcal S\setminus\partial \mathcal S$,
there exist $t_1<0<t_2$ such that $p\cdot (t_1,t_2)\in \mathcal S$,
where $\cdot$ denotes the flow for \eqref{deq_x}.
In the vicinity of a locally invariant manifold,
under certain conditions
the system can be decomposed into lower-dimensional subsystems.
For instance,
when $\mathcal S=\{p_0\}$ is an isolated hyperbolic equilibrium for \eqref{deq_x},
the stable and unstable manifolds $W^s(p_0)$ and $W^u(p_0)$
exist according to the Hartman-Grobman Theorem \cite{Hartman:1964},
and the union of their tangent spaces at $p_0$ spans $\mathbb R^N$.

A locally invariant $C^r$ manifold $\mathcal S\subset \mathbb R^N$, $r\ge 1$,
is normally hyperbolic for the system \eqref{deq_x} if
the growth rate of vectors transverse to the manifold
dominates the growth rate of vectors tangent to the manifold.
(Note that this is consistent with Definition \ref{defn_normally_hyperbolic}.)
In this case, from the standard theory for normally hyperbolic manifolds
(see, for example, \cite{Hirsch:1977,Vanderbauwhede:1987,Chow:1988})
it is assured that
stable and unstable manifolds
$W^s(\mathcal S)$ and $W^u(\mathcal S)$ are defined.

For a locally invariant manifold 
$\Lambda\subset \mathbb R^N$ for \eqref{deq_x}
which is not necessarily normally hyperbolic,
a \emph{center manifold}
is a normally hyperbolic locally invariant manifold,
with the smallest possible dimension,
containing $\Lambda$.
In classical cases,
$\Lambda=\{p_0\}$ is an isolated non-hyperbolic equilibrium,
and a center manifold for $p_0$
has dimension equal to
the number of generalized eigenvalues of $DF(p_0)$
with zero real part.
For instance, the planar system \[
  \dot{x}= x^3,\quad
  \dot{y}= y ,
\] has a non-hyperbolic isolated equilibrium $p_0=(0,0)$,
and the $x$-axis is a center manifold for $p_0$.
For general invariant sets $\Lambda$, we refer to \cite{Chow:2000a,Chow:2000}.

Fenichel's Theory is a center manifold theory for fast-slow systems.
For a normally hyperbolic critical manifold $\mathcal S_0$ for \eqref{fast_xy},
the stable and unstable manifolds
$W^s(\mathcal S_0)$ and $W^u(\mathcal S_0)$
can be defined in the natural way.
We denote them by
$W^s_0(\mathcal S_0)$ and $W^u_0(\mathcal S_0)$
to indicate their invariance under \eqref{sf_xy} with $\epsilon=0$.
Fenichel's Theory assures that
the hyperbolic structure of $\mathcal S_0$ persists under perturbation \eqref{sf_xy}.
Below we state three fundamental theorems
of Fenichel's Theory following \cite{Jones:1995}.

\begin{thm}[Fenichel's Theorem 1]
\label{thm_fenichel_invariant}
Consider the system \eqref{sf_xy},
where $(x,y)\in \mathbb R^n\times \mathbb R^l$,
and $f$, $g$ are $C^r$ for some $r\ge 2$.
Let $\mathcal S_0$ be a compact normally hyperbolic manifold for \eqref{fast_xy}.
Then for any small $\epsilon\ge 0$ there exist
locally invariant $C^r$ manifolds,
denoted by $\mathcal S_\epsilon$,
$W^s_\epsilon(\mathcal S_\epsilon)$
and $W^u_\epsilon(\mathcal S_\epsilon)$,
which are $C^1$ $O(\epsilon)$-close to
$\mathcal S_0$, $W^s_0(\mathcal S_0)$ and $W^u_0(\mathcal S_0)$, respectively.
Moreover, for any continuous families of compact sets
$\mathcal I_\epsilon\subset W^u_\epsilon(\mathcal S_\epsilon)$,
$\mathcal J_\epsilon\subset W^s_\epsilon(\mathcal S_\epsilon)$,
$\epsilon\in [0,\epsilon_0]$,
there exist positive constants $C$ and $\nu$
such that \begin{subequations}\label{est_dist}\begin{align}
  &\mathrm{dist}(z\cdot t,\mathcal S_\epsilon)
  \le Ce^{\nu t}
  \quad\forall\; z\in \mathcal I_\epsilon,\; t\le 0\\
  &\mathrm{dist}(z\cdot t,\mathcal S_\epsilon)
  \le Ce^{-\nu t}
  \quad\forall\; z\in \mathcal J_\epsilon,\; t\ge 0,
\end{align}\end{subequations}
where $\cdot$ denotes
the flow for \eqref{sf_xy}.
\end{thm}

\begin{proof}
See \cite[Theorem 3]{Jones:1995}.
\end{proof}

\begin{rmk}
If $\mathcal S_0$ is locally invariant under \eqref{sf_xy} for each $\epsilon$,
then the $\mathcal S_\epsilon$ can be chosen to be $\mathcal S_0$
because of the construction
in the proof of \cite[Theorem 3]{Jones:1995}.
\end{rmk}

Note that $W^u_\epsilon(\mathcal S_\epsilon)$ and $W^s_\epsilon(\mathcal S_\epsilon)$
can be interpreted as a decomposition
in a neighborhood of $\mathcal S_0$ in $(x,y)$-space.
The following theorem asserts that
this induces a change of coordinates $(a,b,c)$
such that $W^u_\epsilon(\mathcal S_\epsilon)$
and $W^s_\epsilon(\mathcal S_\epsilon)$
correspond to $(a,c)$-space and $(b,c)$-space, respectively.

\begin{thm}[Fenichel's Theorem 2]
\label{thm_fenichel_coordinate}
Suppose the assumptions in Theorem \ref{thm_fenichel_invariant} hold.
Then under a $C^r$ $\epsilon$-dependent
coordinate change $(x,y)\mapsto(a,b,c)$,
the system \eqref{sf_xy} can be brought to the form
\beqeps{sf_abc}
  &\dot a= A^u(a,b,c,\epsilon)a\\
  &\dot b= A^s(a,b,c,\epsilon)b\\
  &\dot c= \epsilon\big( h(c)+ E(a,b,c,\epsilon)\big)
\] in a neighborhood of $\mathcal S_\epsilon$,
where the coefficients are $C^{r-2}$ functions satisfying \beq{cond_spec}
  \inf_{\lambda\in \mathrm{Spec}A^u(a,b,c,0)}
  \mathrm{Re}\,\lambda>2\nu,\quad
  \sup_{\lambda\in \mathrm{Spec}A^s(a,b,c,0)}
  \mathrm{Re}\,\lambda<-2\nu
\] for some $\nu>0$ and \beq{cond_E}
  E=0\quad\text{on }\{a=0\}\cup \{b=0\}.
\] 
\end{thm}

\begin{proof}
See \cite[Section 3.5]{Jones:1995} or \cite[Proposition 1]{Jones:2009}.
\end{proof}

The family of trajectories for \eqref{slow_xy}
forms a foliation of $\mathcal S_0$.
The following theorem says that
this induces a foliation of
$W^u_\epsilon(\mathcal S_\epsilon)$ and $W^s_\epsilon(\mathcal S_\epsilon)$.

\begin{thm}[Fenichel's Theorem 3]
\label{thm_foliation}
Suppose the assumptions in Theorem \ref{thm_fenichel_invariant} hold.
Let $\Lambda_0$ be a submanifold in $\mathcal S_0$
which is locally invariant under \eqref{slow_xy}.
Then there exist locally invariant manifolds $\Lambda_\epsilon$,
$W^s_\epsilon(\Lambda_\epsilon)$,
and $W^u_\epsilon(\Lambda_\epsilon)$
for \eqref{sf_xy}
which are $C^{r-2}$ $O(\epsilon)$-close to
$\Lambda_0$,
$W^s_0(\Lambda_0)$,
and $W^u_0(\Lambda_0)$, respectively.
Moreover, for any continuous families of compact sets
$\mathcal I_\epsilon\subset W^u_\epsilon(\Lambda_\epsilon)$,
$\mathcal J_\epsilon\subset W^s_\epsilon(\Lambda_\epsilon)$,
$\epsilon\in [0,\epsilon_0]$,
there exist positive constants $C$ and $\nu$
such that \eqref{est_dist}
holds with $\mathcal S_\epsilon$ replaced by $\Lambda_\epsilon$.
Suppose in addition that
$S_0$ is invariant under \eqref{sf_xy} for each $\epsilon$.
Then $\Lambda_\epsilon$ can be chosen to be $\Lambda_0$.
\end{thm}

\begin{proof}
Using Fenichel's coordinates $(a,b,c)$
in Theorem \ref{thm_fenichel_coordinate} for the splitting of $\mathcal S_0$,
we can take $W_\epsilon^u(\Lambda_\epsilon)$ and $W_\epsilon^s(\Lambda_\epsilon)$
to be the pre-images of the sets $\{(a,b,c): a=0, c\in \Lambda_0\}$
and $\{(a,b,c): b=0, c\in \Lambda_0\}$, respectively,
in $(x,y)$-space.
From \eqref{cond_spec} we obtain \eqref{est_dist}
with $\mathcal S_\epsilon$ replaced by $\Lambda_\epsilon$.
Suppose $S_0$ is invariant under \eqref{sf_xy} for each $\epsilon$,
then from the remark after Theorem \ref{thm_fenichel_invariant},
we can take $\mathcal S_\epsilon=\mathcal S_0$
and hence $\Lambda_\epsilon=\Lambda_0$
\end{proof}

The system \eqref{sf_abc} is called a \emph{Fenichel normal form} for \eqref{sf_xy},
and the variables $(a,b,c)$ are called \emph{Fenichel coordinates}.

\subsection{Silnikov Boundary Value Problem} \label{subsec_silnikov}
\begin{figure}[t]
\centering
{
\includegraphics[trim = 4.5cm 8.5cm 4.5cm 6.5cm, clip, width=.33\textwidth]{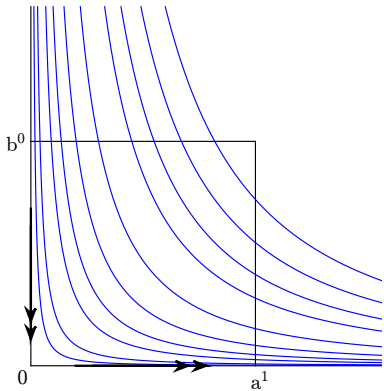}
}
\caption{
Trajectories in the rectangle $\{0\le a\le a^1,0\le b\le b^0\}$
can be parametrized in $T\ge 0$ by $a(T)=a^1$, $b(0)=b^0$.
}
\label{fig_silnikov}
\end{figure}

We have seen in Section \ref{subsec_fenichel}
that fast-slow systems \eqref{sf_xy}
can locally be converted into normal forms \eqref{sf_abc},
where $A^u$ and $A^s$ satisfy the gap condition \eqref{cond_spec},
and $E$ is a small term satisfying \eqref{cond_E}.
If we append the system with the equation $\dot\epsilon=0$
and then replace $c$ by $\tilde c=(c,\epsilon)$,
we obtain a system of the form \beq{deq_abc_center}
  &\dot{a}= A^u(a,b,\tilde c)a\\
  &\dot{b}= A^s(a,b,\tilde c)b\\
  &\dot{\tilde c}= \tilde h(\tilde c)+ E(a,b,\tilde c),
\] for which \eqref{cond_spec} and \eqref{cond_E} are satisfied
with $E$ replaced by $\tilde E$.
For convenience, we will drop the tilde notation in \eqref{deq_abc_center}
in the remaining discussion.

A Silnikov problem is the system \eqref{deq_abc_center}
along with boundary data of the form \beq{bc_silnikov}
  (b,c)(0)=(b^0,c^0),\quad
  a(T)=a^1,
\] where $T\ge 0$.
This boundary value problem was posed in \cite{Silnikov:1967}
to study homoclinic bifurcation.
A heuristic reason for the existence of solutions of a Silnikov problem
is illustrated in Fig \ref{fig_silnikov}.
Consider the simple case $\dot{a}=a$, $\dot{b}=-b$ and $\dot c=0$.
There are infinitely many trajectories
contained in the box $\{0\le a\le a^1,0\le b\le b^0\}$.
We may parametrize the set of trajectories in $T\ge 0$ by $b(0)=b^0$ and $a(T)=a^1$.
On the $a$-axis and $b$-axis,
the trajectories tend to the origin in backward and forward time, respectively.
This suggests that trajectories
near the axes can stay for an arbitrarily long time in the box,
which implies that for any large $T$
there exists a trajectory satisfying $b(0)=b^0$ and $a(T)=a^1$.
When $T$ grows to infinity,
the trajectories approach the axes.
In the general case $\dot a=A^ua$ and $\dot b=A^sb$ in arbitrary dimension,
both $a$- and $b$-spaces consist of solutions tending to the origin
in forward or backward time,
so we have the same conclusion.

The critical manifold for \eqref{deq_abc_center} is $\{a=0,b=0\}$,
on which the system is governed by
the limiting slow system \beq{deq_c_critical}
  \dot{c}= h(c).
\] For a solution $(a(t),b(t),c(t))$ to
the Silnikov boundary value problem \eqref{deq_abc_center} and \eqref{bc_silnikov},
from conditions \eqref{cond_spec} and \eqref{cond_E},
it is natural to expect that
$a(t)$ and $b(t)$ decay to $0$ in backward time and forward time, respectively,
and that $c(t)$ is approximately the solution of \eqref{deq_c_critical}.
A theorem from \cite{Schecter:2008a}
asserts that this is the case:

\begin{thm}[Generalized Deng's Lemma \cite{Schecter:2008a}]
\label{thm_bvp_h}
Consider the system \eqref{deq_abc_center}
satisfying \eqref{cond_spec} and \eqref{cond_E}
with $C^r$ coefficients, $r\ge 1$,
defined on the closure of a bounded open set
$B_{k,\Delta}\times B_{m,\Delta}\times V
\subset \mathbb R^k\times \mathbb R^m\times \mathbb R^l$,
where $B_{k,\Delta}=\{a\in\mathbb R^k: |a|<\Delta\}$,
$\Delta>0$, and $V$ is a bounded open set in $\mathbb R^l$.

Let $K_0$ and $K_1$ be compact subsets of $V$
such that $K_0\subset \mathrm{Int}(K_1)$.
For each $c^0\in K_0$ let $J_{c^0}$ be the maximal interval
such that $\phi(t,c^0)\in \mathrm{Int}(K_1)$ for all $t\in J_{c^0}$,
where $\phi(t,c^0)$ is the solution of \eqref{deq_c_critical} with initial value $c^0$.
Let $\nu>0$ be the number in \eqref{cond_spec}.
Suppose there exists $\beta>0$ such that $\tilde \nu:= \nu- r\beta>0$ and \[
  |\phi(t,c^0)|\le Me^{\beta|t|}\quad\forall\; t\in J_{c^0}.
\] Then there is a number $\delta_0>0$
such that if $|a^1|<\delta_0$, $|b^0|<\delta_0$, $c^0\in V_0$, and ${T}>0$ is in $J_{c^0}$,
then the Silnikov boundary value problem \eqref{deq_abc_center} and \eqref{bc_silnikov}
has a solution $(a,b,c)(t,{T},a^1,b^0,c^0)$ on the interval $0\le t\le {T}$.
Moreover, there is a number $K>0$ such that for all $(t,{T},a^1,b^0,c^0)$ as above
and for all multi-indices $\mathbf{i}$ with $|\mathbf{i}|\le r$, \beq{est_Di_bvp}
  &|D_{\mathbf{i}}a(t,{T},a^1,b^0,c^0)|\le Ke^{-\tilde\nu ({T}-t)}\\
  &|D_{\mathbf{i}}b(t,{T},a^1,b^0,c^0)|\le Ke^{-\tilde\nu t}\\
  &|D_{\mathbf{i}}c(t,{T},a^1,b^0,c^0)-D_{\mathbf{i}}\phi(t,c^0)|
  \le Ke^{-\tilde\nu T}.
\]
\end{thm}

\begin{proof}[Sketch of Proof] 
Here we sketch the proof in \cite{Schecter:2008}.
Write \eqref{deq_abc_center} as \[
  &\dot{a}= \tilde{A}^u(t,c^0)a+ f(t,c^0,a,b,z)\\
  &\dot{b}= \tilde{A}^s(t,c^0)b+ g(t,c^0,a,b,z)\\
  &\dot{z}= \tilde{A}^c(t,c^0)z+ \theta(t,c^0,z)+ \tilde{E}(t,c^0,a,b,z),
\] where \[
  &\tilde{A}^i(t,c^0)= {A}^i(0,0,\phi(t,c^0)),\quad i=u,s,\\
  &\tilde{A}^c(t,c^0)= Dh\big|_{\phi(t,c^0)}
\] and \[
  \tilde{E}(t,c^0,a,b,z)= E(a,b,\phi(t,c^0)+z).
\] Let $\Phi^i(t,s,c^0)$ be the solution operator for $\tilde A^i(t,c^0)$, $i=u,s,c$.
Then $(a(t),b(t),c(t))$ is a solution of Silnikov problem \eqref{deq_abc_center} and \eqref{bc_silnikov}
if and only if $c(t)= \phi(t,c^0)+z(t)$
and $\eta(t)= (a(t),b(t),z(t))$ satisfies \beq{eq_abz_integral}
  &a(t)= \Phi^u(t,T,c^0)a^1- \int_t^T \Phi^u(t,s,c^0)f(s,c^0,\eta(s))\;ds\\
  &b(t)= \Phi^s(t,0,c^0)b^0+ \int_0^t \Phi^s(t,s,c^0)g(s,c^0,\eta(s))\;ds\\
  &z(t)= \int_0^t \Phi^c(t,s,c^0)\big(\theta(s,c^0,z(s))+ \tilde{E}(s,c^0,\eta(s))\big)\;ds.
\] Define an linear operator $\mathcal L$
by the right-hand side of \eqref{eq_abz_integral}
for functions $\eta(t)= (a(t),b(t),z(t))$.
It can be shown that the restriction of $\mathcal L$
on a neighborhood of $0$
in the space of functions $\eta(t)= (a(t),b(t),z(t))$ equipped with the norm \[
  \|\eta\|_j
  = \sup_{0\le t\le T}
  \big(
    e^{\tilde{\nu}(T-t)}|a(t)|+ e^{\tilde{\nu}t}|b(t)|+ e^{\tilde{\nu}T}|z(t)|
  \big)
\] is a contraction mapping.
Hence the existence of solution of \eqref{deq_abc_center} and \eqref{bc_silnikov}
follows from the standard Banach fixed point theorem.
\end{proof}

\begin{rmk}
Theorem \ref{thm_bvp_h}
is a generalization of the \emph{Strong $\lambda$-Lemma} in Deng \cite{Deng:1990},
and \emph{$C^r$-Inclination Theorem} in Brunovsky \cite{Brunovsky:1999}.
In Deng's work,
the boundary data lie near an equilibrium that may nonhyperbolic.
In Brunovsky's work,
the boundary data lie near a solution of a rectifiable slow flow
on a normally hyperbolic invariant manifold.
Schecter's work allows
considering more general flows on normally hyperbolic invariant manifolds.
\end{rmk}

\subsection{The Exchange Lemma} \label{subsec_EL}
Consider \eqref{sf_abc} as a special case of \eqref{deq_abc_center},
and recall that \eqref{sf_abc} is the normal form
of fast-slow systems \eqref{sf_xy}.
We will use Theorem \ref{thm_bvp_h}
to analyze Silnikov problems for fast-slow systems \eqref{sf_xy}.
The result turns out to be 
a variation of the $(k+\sigma)$-Exchange Lemma \cite{Jones:2009,Tin:1994}.

The Silnikov problem for \eqref{sf_abc} corresponds to the boundary data \beq{bc_abc_eps}
  a(\tau/\epsilon)= a^1,\quad
  (b,c)(0)= (b^0,c^0),
\] with given $(a^1,b^0,c^0)\in \mathbb R^k\times \mathbb R^m\times \mathbb R^l$
and $\tau>0$.
It can be interpreted as finding trajectories for \eqref{sf_abc}
connecting the sets $\{b=b^0,c=c^0\}$ and $\{a=a^1\}$,
with prescribed time interval $0\le t\le \tau/\epsilon$;
see Fig \ref{fig_EL_abc}.
Note that the set $\{b=b^0,c=c^0\}$ is of dimension $k$.
The Exchange Lemma is a tool tracking the $(k+1)$-manifold $\mathcal I_\epsilon^*$
that evolves from a $k$-manifold $\mathcal I_\epsilon$
which is transverse to the center-stable manifold $\{a=0\}$.
The theory of Exchange Lemma
was first developed in \cite{Jones:1991,Jones:1994,Jones:1996}
to study singularly perturbed systems
near a normally hyperbolic, locally invariant manifold.
Some generalizations of the Exchange Lemma
for a broader class of systems
were given by W.\ Liu \cite{Liu:2000} and Schecter \cite{Schecter:2008a}.

Another generalization,
given by Tin \cite{Tin:1994},
is the $(k+\sigma)$-Exchange Lemma, $1\le \sigma\le l$,
which tracks the $(k+\sigma)$-manifold $\mathcal I_\epsilon^*$
which evolves from a $(k+\sigma-1)$-manifold $\mathcal I_\epsilon=\{b=b^0,c^0\in \Lambda\}$,
where $\Lambda$ is a $(\sigma-1)$-manifold.
A major difference between the $(k+\sigma)$-Exchange Lemma
and the general Exchange Lemma in \cite{Schecter:2008a}
is that the estimates \eqref{est_Di_bvp} for the derivatives in slow variables
were not considered in \cite{Schecter:2008a}.

We analyze Silnikov problems
for fast-slow systems in normal form \eqref{sf_abc} in Lemma \ref{lem_bvp_EL},
and then, in Theorem \ref{thm_EL}, return to \eqref{sf_xy}
to present a version of the $(k+\sigma)$-Exchange Lemma.

\begin{figure}[t]
\centering
{
\includegraphics[trim = 7cm 7.5cm 3.4cm 9cm, clip, width=.4\textwidth]{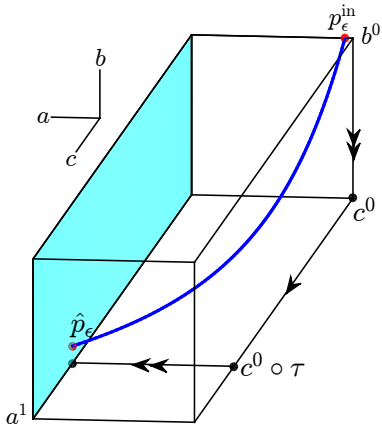}
}
\caption{
The entry point $(a^\din_\epsilon,b^0,c^0)$
is close to $(0,b^0,c^0)$,
and the exit point $(a^1,\hat{b}_\epsilon,\hat{c}_\epsilon)$
is close to $(a^1,0,c^0\circ\tau)$, as $\epsilon\to 0$.
}
\label{fig_EL_abc}
\end{figure}

\begin{lem}\label{lem_bvp_EL}
Consider a system of the form \eqref{sf_abc}
satisfying \eqref{cond_spec} and \eqref{cond_E}
defined on the closure of 
a bounced open set
$B_{k,\Delta}\times B_{m,\Delta}\times V
\subset \mathbb R^k\times\mathbb R^m\times\mathbb R^l$,
where the coefficients are $C^r$ for some integer $r\ge 0$.
Let $\Lambda\subset V$ be a $(\sigma-1)$-dimensional manifold, $1\le \sigma\le l$
and $\tau_0>0$. Suppose \[
  c\circ [0,\tau_0]\subset V\quad
  \forall\; c\in \Lambda,
\] where $\circ$ denotes the flow for the limiting slow system \eqref{slow_xy}.
Let $J\subset (0,\tau_0)$ be a closed interval
and $\mathcal A\subset B_{k,\Delta}\setminus\{0\}$ be a compact set.
Then for each small $\epsilon>0$
and $(a^1,c^0,\tau)\in \mathcal A\times \Lambda\times J$,
the boundary value problem \eqref{sf_xy} and \eqref{bc_abc_eps}
has a unique solution,
denoted by $(a,b,c)(t;\tau,a^1,b^0,c^0,\epsilon)$, $t\in [0,\tau/\epsilon]$.
Moreover, if we set \beq{def_peps}
  p_\epsilon= (a,b,c)(0;\tau,a^1,b^0,c^0,\epsilon),\quad
  q_\epsilon= (a,b,c)(\tau/\epsilon;\tau,a^1,b^0,c^0,\epsilon),
\] then \beq{est_pin_EL_abc}
  \|p_\epsilon- (0,b^0,c^0)\|_{C^{r}(\mathcal A\times\Lambda\times J)}
  + \|{q}_\epsilon- (a^1,0,c^0\circ \tau)\|_{C^{r}(\mathcal A\times\Lambda\times J)}
  \le Ce^{-\tilde\nu/\epsilon}
\] for some positive constants $\tilde C$ and $\tilde\nu$.
See Fig \ref{fig_EL_abc}.
\end{lem}

\begin{proof}[Sketch of Proof]
Existence of solutions follows directly from Theorem \ref{thm_bvp_h},
so it remains to prove \eqref{est_pin_EL_abc}.
Write $p_\epsilon=(a^\din_\epsilon,b^0,c^0)$
and $q_\epsilon=(a^1,\hat{b}_\epsilon,\hat{c}_\epsilon)$,
then \eqref{est_pin_EL_abc} is equivalent to \beq{est_abc_EL}
  \|(a^\din_\epsilon,\hat{b}_\epsilon,\hat{c}_\epsilon-c^0\circ \tau)\|
  {}_{C^r(\mathcal A\times \Lambda\times J)}
  \le \tilde C e^{-\tilde\nu/\epsilon}.
\]
The estimate of the derivatives in $(a^1,c^0)\in \mathcal A\times \Lambda$ in \eqref{est_abc_EL}
follows directly from \eqref{est_Di_bvp}.
To prove the estimate of the derivatives in $\tau\in J$,
note that from \eqref{eq_abz_integral} we have \beq{hat_a_integral}
  &a^\din_\epsilon= \Phi^u(0,\tau/\epsilon,c^0)a^1
  - \int_0^{\tau/\epsilon} \Phi^u(0,s,c^0)f(s,c^0,\eta(s))\;ds\\
  &\hat{b}_\epsilon= \Phi^s(\tau/\epsilon,0,c^0)b^0
  + \int_0^{\tau/\epsilon} \Phi^s(\tau/\epsilon,s,c^0)g(s,c^0,\eta(s))\;ds\\
  &\hat{c}_\epsilon= c^0\circ \tau
  + \int_0^{\tau/\epsilon} \Phi^c(\tau/\epsilon,s,c^0)\big(\theta(s,c^0,z(s))+ \tilde{E}(s,c^0,\eta(s))\big)\;ds.
\] As in the proof of Theorem \ref{thm_bvp_h},
it can be shown that the derivatives of the integrands in \eqref{hat_a_integral} are exponentially small,
so we obtain \eqref{est_abc_EL}.
\end{proof}

\begin{figure}[t]
\centering
{
\includegraphics[trim = 7cm 7.5cm 3.4cm 9cm, clip, width=.46\textwidth]{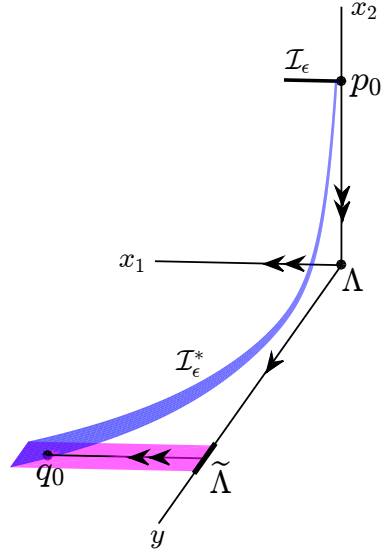}
}
\caption{
The $(k+\sigma)$-Exchange Lemma asserts that
$\mathcal I_\epsilon^*$ is $C^1$-close to 
$W^u_0(\tilde\Lambda)$
in a neighborhood of $q_0$.
}
\label{fig_EL_xy}
\end{figure}

The following theorem is a modification
of the $(k+\sigma)$-Exchange Lemma.
The main difference 
is that in this version we assert the existence of certain trajectories,
while in the original version those trajectories were assumed to exist.
The proof of the original theorem \cite{Tin:1994}
is based on tracking tangent spaces to an invariant manifold
using linearized differential equations in terms of differential forms,
while the approach we present below relies on estimates for solution operators,
following closely to the proof of
the general Exchange Lemma in \cite{Schecter:2008}.

\begin{thm}
\label{thm_EL}
Consider a system of the form \eqref{sf_xy}
where $(x,y)\in \mathbb R^n\times\mathbb R^l$,
and $f$ and $g$ are $C^r$ functions for some $r\ge 2$.
Let $\mathcal S_0$ be a normally hyperbolic critical manifold for \eqref{fast_xy},
and suppose $D_xf|_{\mathcal S^0}$
has a splitting of $k$ unstable eigenvalues and $m$ stable eigenvalues, $k+m=n$.
Let $\bar{q}_0\in W^u_0(\mathcal S_0)\setminus \mathcal S_0$,
$\bar{p}_0\in W^s_0(\mathcal S_0)\setminus \mathcal S_0$, $\bar{\tau}_0>0$,
and assume \beq{link_pq0_EL}
  \pi^s(\bar{p}_0)\circ [0,\bar{\tau}_0]
  \subset S_0
  \quad\text{and}\quad
  \pi^u(\bar{q}_0)
  = \pi^s(\bar{p}_0)\circ \tau_0,
\] where $\circ$ denotes the flow for the limiting slow system \eqref{slow_xy},
and $\pi^{s,u}$ are the projections into $\mathcal S_0$ along stable/unstable fibers
with respect to the limiting fast system \eqref{fast_xy}.
Let $\{\mathcal I_\epsilon\}_{\epsilon\in [0,\epsilon_0]}$
be a $C^r$ family of $(k+\sigma-1)$-dimensional manifolds, $1\le \sigma\le l$,
and suppose \begin{enumerate}
  \item[$\mathrm{(T1)}$]
  $\mathcal I_0$ is transverse to $W^s(\mathcal S_0)$ at $p_0$,
  and $\Lambda:=\pi^s(\mathcal I_0\cap W^s(\mathcal S_0))$
  is of dimension $(\sigma-1)$.
   \item[$\mathrm{(T2)}$]
   the slow flow \eqref{slow_xy} is not tangent to $\Lambda$ at $\pi^s(\bar{p}_0)$.
  \item[$\mathrm{(T3)}$]
  The trajectory $\pi^s(p_0)\circ [0,\tau_0]$ is rectifiable and not self-intersecting.
\end{enumerate}
Let \beq{def_Istar_EL}
  \mathcal I_\epsilon^*= \mathcal I_\epsilon\cdot [0,\infty),
\] where $\cdot$ denotes the flow for \eqref{sf_xy}.
Choose a compact interval $J\subset (0,\infty)$ containing $\bar{\tau}_0$
satisfying $\Lambda\circ J\subset \mathcal S_0$,
and set $\widetilde\Lambda=\Lambda\circ J$.
Then there exists a neighborhood $V_0$ of $\bar{q}_0$
such that \beq{closeness_EL}
  \text{
    $\mathcal I_\epsilon^*\cap V_0$
    is $C^{r-2}$ $O(\epsilon)$-close to $W^u_0(\widetilde\Lambda)\cap V_0$.
  }
\] See Fig \ref{fig_EL_xy}.
Moreover, given any sequence $\bar{q}_\epsilon\in \mathcal I_\epsilon^*\cap V_0$
such that $\bar{q}_\epsilon\to \bar{q}_0$,
there exists a sequence
$(\bar{p}_\epsilon,\bar{\tau}_\epsilon)\in \mathcal I_\epsilon\times J$
which converges to $(\bar{p}_0,\bar{\tau}_0)$ and satisfies that,
setting $T_\epsilon=\bar{\tau}_\epsilon/\epsilon$, \beq{link_pq_EL}
  \bar{q}_\epsilon=\bar{p}_\epsilon\cdot T_\epsilon
  \quad\forall\; \epsilon>0,
\] and \beq{est_Teps_EL}
  T_\epsilon= \big(
    \tau_0+o(1)
  \big)\epsilon^{-1}.
\]
\end{thm}

\begin{proof}
Under the assumption \eqref{cond_E_weak},
from \cite[Lemma 2.2]{Deng:1990},
after a $C^{r-2}$ change of of coordinates,
we can convert \eqref{sf_xy} to \eqref{sf_abc},
and, from $\mathrm{(T1)}$, we may assume \beq{Ieps_normal_EL}
  \mathcal I_\epsilon
  = B_{k,\Delta}\times \{b^0\}\times \Lambda
\] for some constant $b^0\in B_{m,\Delta}\setminus\{0\}$.

Since $\bar{q}_0\in W^u_0(\mathcal S_0)\setminus \mathcal S_0$,
we have $a(\bar{q}_0)\ne 0$ and $b(\bar{q}_0)= 0$,
where $a(\bar{q}_0)$ and $b(\bar{q}_0)$
denote the $a$- and $b$-coordinates of $\bar{q}_0$.
Set \beq{def_A_EL}
  \mathcal A=\{a\in \mathbb R^k: |a-a(\bar{q}_0)|<\Delta_1\}
\] for some positive number $\Delta_1<\frac12\min\{\Delta,|a(\bar{q}_0)|\}$,
so that $\mathcal A\subset B_{k,\Delta}\setminus\{0\}$.
Let $p_\epsilon$ and $q_\epsilon$ be
the functions of $(a^1,c^0,\tau)\in \mathcal A\times \Lambda\times J$
defined by \eqref{def_peps}.
From \eqref{Ieps_normal_EL} we see that $(p_\epsilon,\tau)$
parametrizes $\mathcal I_\epsilon\times J$ in a neighborhood of $(p_0,\tau_0)$.
Hence $q_\epsilon$
parametrizes $\mathcal I_\epsilon^*$
in neighborhoods of $\bar{q}_0$.
The estimate \eqref{est_pin_EL_abc} holds with $r$ replaced by $r-2$.
In particular, \[
  \|{q}_\epsilon- (a^1,0,c^0\circ \tau)\|_{C^{r-2}(\mathcal A\times\Lambda\times J)}
  \le Ce^{-\tilde\nu/\epsilon}.
\]
Note that \beq{WutLambda_EL}
  W^u(\widetilde\Lambda)
  &= \{(a,b,c): b=0, c\in \widetilde\Lambda\}\\
  &= \{(a,b,c^0\circ \tau): b=0, c^0\in \Lambda,\tau\in J\},
\] so we obtain \eqref{closeness_EL}.

Next we consider the sequence $\bar{q}_\epsilon\in \mathcal I_\epsilon$ given in the statement.
Choose $(a^1_\epsilon,c^0_\epsilon,\tau_\epsilon)\in \mathcal A\times \Lambda\times J$
such that $\bar{q}_\epsilon= {q}_\epsilon(a^1_\epsilon,c^0_\epsilon,\tau_\epsilon)$,
and set $\bar{p}_\epsilon= {p}_\epsilon(a^1_\epsilon,c^0_\epsilon,\tau_\epsilon)$.
Then by definition $\bar{q}_\epsilon= \bar{p}_\epsilon\cdot (\bar{\tau}_\epsilon/\epsilon)$.
From $\mathrm{(T2)}$ and $\mathrm{(T3)}$,
$\widetilde\Lambda$
is a $\sigma$-dimensional manifold,
and for any $c^1\in \widetilde\Lambda$,
there exists unique $(c^0,\tau_0)\in \Lambda\times J$
such that $c^1=c^0\circ \tau_0$.
Hence \eqref{link_pq_EL} uniquely determines
$\bar{p}_0\in \mathcal I_0\cap W^s(\Lambda)$ and $\bar{\tau}_0\in J$.
To show $p_\epsilon\to \bar{p}_0$ and $\tau_\epsilon\to \bar{\tau}_0$,
since $(p_\epsilon,\tau_\epsilon)$ lies in the compact set $\Lambda\times J$,
it suffices to show that
every convergent subsequence
of $\{(c_{\epsilon},\tau_{\epsilon})\}$
converges to $(\bar{c}^0,\bar{\tau}_0)$.
Note that from the equation for $\hat{c}_\epsilon$ in \eqref{hat_a_integral},
we have \beq{est_c01_EL}
  \hat{c}_\epsilon
  = c^0_\epsilon\circ \tau_\epsilon+ o(1).
\]
Since $q_\epsilon\to \bar{q}_0\equiv (\bar{a}^1,0,\bar{c}^1)$,
given any convergent subsequence $(c_{\epsilon j},\tau_{\epsilon j})$
of $(c_{\epsilon},\tau_{\epsilon})$,
say $(c_{\epsilon j},\tau_{\epsilon j})\to (\tilde{c}^0,\tilde{\tau}_0)$,
from \eqref{est_c01_EL} we obttain $\bar{c}^1= \tilde{c}^0\circ \tilde{\tau}_0$.
From \eqref{link_pq0_EL} we have $\bar{c}^1= \bar{c}^0\circ \bar{\tau}_0$.
Hence $(\tilde{c}^0,\tilde{\tau}_0)=(\bar{c}^0,\bar{\tau}_0)$.
This completes the proof.
\end{proof}

\section{Singular Configuration}
\label{sec_singular_config}
The fast-slow system \eqref{sf_u}
has multiple limiting subsystems corresponding to different time scales.
In this section we will find trajectories, called \emph{singular trajectories}, for those subsystems
such that the union of those trajectories joins the end states $u_L$ and $u_R$.
The union of those singular trajectories is called a \emph{singular configuration}.
In later sections we will show that there are solutions of \eqref{sf_u}
close to the singular configuration.

\subsection{End States $\mathcal U_L$ and $\mathcal U_R$}
\label{subsec_fenichel_UL}
The system \eqref{fast_bv}
has a normally hyperbolic critical manifold \beq{def_s0}
  \mathcal S_0
  = \big\{
    (u,w,\xi):
    f(u)- \xi u - w= 0,
    \xi\ne \mathrm{Re }(\lambda_\pm(u))
  \big\},
\] where $\lambda_{\pm}(u)$ are the eigenvalues of $Df(u)$,
defined in \eqref{def_lambda}.
The limiting slow system for \eqref{sf_bv} is \beq{slow_u}
  &0=f(u)-\xi u- w\\
  &w'=-u\\
  &\xi'=1.
\]
From $\mathrm{(H1)}$
we have $s<\mathrm{Re}(\lambda_\pm(u_L))$,
so $(u_L,w_L,s)\in \mathcal S_0$.
Choose $\delta>0$ so that $s+2\delta<\mathrm{Re}(\lambda_\pm(u_L))$,
and set \beq{def_ul}
  \mathcal U_L
  &= (u_L,w_L,s)\myflow{slow_u} (-\infty,\delta]\\
  &= \{(u,w,\xi): u=u_L, w=w_L-\alpha_1 u_L, \xi= s+\alpha_1, \alpha_1\in(-\infty,\delta]\},
\] where $\myflow{slow_u}$
denotes the flow for \eqref{slow_u}.
It is clear that $\mathcal U_L\subset \mathcal S_0$
is normally hyperbolic with respect to \eqref{fast_bv},
and is locally invariant with respect to \eqref{sf_u}.

Note that each point in $\mathcal U_L$
is a hyperbolic equilibrium for the $2$-dimensional system \eqref{fast_u},
and the unstable manifold $W^u_0(\mathcal U_L)$ is naturally defined.

\begin{prop}\label{prop_ul}
Assume $\mathrm{(H1)}$.
Let $\mathcal U_L$ be defined by \eqref{def_ul}.
Fix any $k\ge 1$.
There exists a family of invariant manifolds
$W_\epsilon^u(\mathcal U_L)$
which are $C^k$ $O(\epsilon)$-close to $W_0^u(\mathcal U_L)$
such that for any continuous family $\{\mathcal I_\epsilon\}_{\epsilon\in [0,\epsilon_0]}$ of
compact sets $\mathcal I_\epsilon\subset W_\epsilon^u(\mathcal U_L)$, \beq{est_dist_ul}
  \mathrm{dist}(p\myflow{sf_u} t,\mathcal U_L)
  \le Ce^{\mu t}
  \quad\forall\; p\in \mathcal I_\epsilon,\; t\le 0, \epsilon\in [0,\epsilon_0],
\] for some positive constants $C$ and $\mu$.
\end{prop}

\begin{proof}
This follows
from Theorem \ref{thm_foliation}
by taking $\mathcal U_L$ to be $\mathcal{U}_0$.
Although $\mathcal U_L$ is not compact,
it is uniformly normally hyperbolic 
since $\xi-\mathrm{Re}(\lambda_\pm(u_L))<-\delta$ on $\mathcal U_L$,
and the proof of Theorem \ref{thm_foliation}
in \cite[Theorem 4]{Jones:1995} is still valid.
\end{proof}

\begin{rmk}
Proposition \ref{prop_ul}
was also asserted in \cite{Schecter:2004,Liu:2004,Keyfitz:2012}.
\end{rmk}

From $\mathrm{(H1)}$ we also have,
by decreasing $\delta$ if necessary,
$s-2\delta>\mathrm{Re}(\lambda_\pm(u_R))$,
and hence a similar result holds for the set $\mathcal U_R$ defined by
\beq{def_ur}
  \mathcal U_R
  &= (u_R,w_R,s)\myflow{slow_u} [-\delta,\infty)\\
  &= \{(u,w,\xi): u=u_R, w=w_R-\alpha_2 u_R, \xi= s+ \alpha_2, \alpha_2\in [-\delta,\infty)\}.
\]

\begin{prop}
Assume $\mathrm{(H1)}$.
Let $\mathcal U_R$ be defined by \eqref{def_ur}.
Fix any $k\ge 1$.
There exists a family of invariant manifolds
$W_\epsilon^s(\mathcal U_R)$
which are $C^k$ $O(\epsilon)$-close to $W_0^s(\mathcal U_R)$
such that for any continuous family
$\{\mathcal J_\epsilon\}_{\epsilon\in [0,\epsilon_0]}$ 
of compact sets
$\mathcal J_\epsilon\subset W^s_\epsilon(\mathcal U_R)$,
\beq{est_dist_ur}
  \mathrm{dist}(p\myflow{sf_u} t,\mathcal U_R)
  \le Ce^{-\mu t}
  \quad\forall\; p\in \mathcal J_\epsilon,\; t\ge 0, \epsilon\in [0,\epsilon_0],
\] for some positive constants $C$ and $\mu$.
\end{prop}

\subsection{Intermediate States $\mathcal P_L$ and $\mathcal P_R$} \label{subsec_PL}
Consider the system \eqref{sf_bv}.
In order to study the dynamics at $\{v=+\infty\}$,
we set $r=1/v$ and $\kappa= \epsilon \log(1/r)$.
Then \eqref{sf_bv} is converted, after multiplying the equations by $r$, to 
\beqeps{sf_brk}
  &\dot{\beta}= B_1(\beta)-\xi\beta r-w_1 r\\
  &\dot{r}= -r B_2(\beta)+ \xi r^2+ w_2r^3\\
  &\dot{w_1}=-\epsilon\beta r\\
  &\dot{w_2}=-\epsilon\\
  &\dot{\xi}=\epsilon r\\
  &\dot\kappa= \epsilon \big(
    B_2(\beta)+ \xi r+ w_2r^2
  \big).
\] 
Note that the time variable in \eqref{sf_brk}
is different from that of \eqref{sf_bv}.
We use the same dot symbol to denote derivatives,
but there should be no ambiguity
since the different time scales can be distinguished by
comparing the term $\dot\xi$.

The limiting fast system for \eqref{sf_brk} is \beq{fast_brk}
  &\dot{\beta}= B_1(\beta)-\xi\beta r-w_1 r\\
  &\dot{r}= -r B_2(\beta)+ \xi r^2+ w_2r^3\\
  &\dot{w_1}=0,\;
  \dot{w_2}=0,\;
  \dot{\xi}=0,\;
  \dot{\kappa}=0.
\]
The obvious equilibria for \eqref{fast_brk},
besides $(\beta_L,r_L,w_{1L},w_{2L},s)$
and $(\beta_R,r_R,w_{1L},w_{2L},s)$,
where $r_L=1/v_L$ and $r_R=1/v_R$,
are \begin{align}
  &\mathcal P_L= \{(\beta,r,w_1,w_2,\xi,\kappa): \beta=\rho_1, r=0\},\label{def_PL}\\
  &\mathcal P_R= \{(\beta,r,w_1,w_2,\xi,\kappa): \beta=\rho_2, r=0\}.\label{def_PR}
\end{align} The limiting slow system on $\mathcal P_L$ is \beq{slow_PL}
  w_1'= 0,\;
  w_2'= -1,\;
  \xi'= 0,\;
  \kappa'= B_2(\rho_1),
\] and on $\mathcal P_R$ is \beq{slow_PR}
  w_1'= 0,\;
  w_2'= -1,\;
  \xi'= 0,\;
  \kappa'= B_2(\rho_2)
\] The Fenichel coordinates near $\mathcal P_L$ can be described as follows.

\begin{prop}\label{prop_pl}
Let $W^{u,s}_\epsilon(\mathcal P_L)$ be the $C^k$ unstable/stable manifolds
of $\mathcal P_L$ for \eqref{sf_brk}, $k\ge 1$.
Then there exists a $C^k$ function
${\hat\beta}= {\hat\beta}(\beta,r,w_1,w_2,\xi,\kappa,\epsilon)$
such that \beq{tbeta0}
  {\hat\beta}= \beta
  \quad\text{when }r=0
\] and $({\hat\beta},r,w_1,w_2,\xi,\kappa)$
is a change of coordinates near $\mathcal P_L$ satisfying \begin{align}
  &W^s_\epsilon(\mathcal P_L)=
  \{
    ({\hat\beta},r,w_1,w_2,\xi,\kappa): {\hat\beta}= \rho_1
  \}\label{WsPL}\\
  &W^u_\epsilon(\mathcal P_L)=
  \{
    ({\hat\beta},r,w_1,w_2,\xi,\kappa): r= 0
  \}\label{WuPL}.
\end{align}
Moreover, the projection $\pi^s_{\epsilon,\mathcal P_L}$ into $\mathcal P_L$
along stable fibers with respective to \eqref{sf_brk} is \beq{pi_s_PL}
  \pi^s_{\epsilon,\mathcal P_L}(\rho_1,r,w_1,w_2,\xi,\kappa)
  = (\rho_1,0,w_1,w_2,\xi,\kappa)
\] in $({\hat\beta},r,w_1,w_2,\xi,\kappa)$-coordinates.
\end{prop}

\begin{proof}
The linearization of \eqref{fast_brk} at $\mathcal P_L$
corresponds to the matrix \beq{linear_PL}
  \begin{pmatrix}
    B_1'(\rho_1)& -\xi \rho_1\\
    0& -B_2(\rho_1)
  \end{pmatrix}
  =\begin{pmatrix}
    1-\frac{\rho_2}{\rho_1}& -\xi \rho_1\\
    0& -\tfrac12\big(1-\frac{\rho_2}{\rho_1}\big)
  \end{pmatrix}
\] which has one positive and one negative eigenvalue.
Note that $\mathcal P_L$ is invariant under \eqref{sf_brk} for each $\epsilon$.
From Theorem \ref{thm_fenichel_invariant} and the remark following it,
${W}^s_\epsilon(\mathcal P_L)$ and ${W}^u_\epsilon(\mathcal P_L)$
are well defined and both have dimension $1$,
and we may take ${W}^u_\epsilon(\mathcal P_L)=\{r=0\}$.
Note that $\{\beta=\rho_1\}$ is transverse to ${W}^s_\epsilon(\mathcal P_L)$,
so we can choose Fenichel coordinates $(a,b,c)$
corresponding to this splitting with $b=r$ and \[
  a= \beta- \rho_1+ \phi(w_1,w_2,\xi,\kappa,\epsilon,r)r
\] for some $C^k$ function $\phi$.
Let ${\hat\beta}=a+\rho_1$.
Then the desired result follows.
\end{proof}

An analogous result holds for $\mathcal P_R$. We omit it here.

\subsection{Transversal Intersections} \label{subsec_transversal}
Fix small $r^0>0$ such that $\gamma_1$ intersects $\{r=r^0\}$
at a unique point. Denote this point by $p^\din_0$. That is, \beq{def_pin0}
  p^\din_0=\gamma_1\cap \{r=r^0\}.
\] We set \beq{def_Ieps}
  \mathcal I_\epsilon
  =W^u_\epsilon(\mathcal U_L)\cap \{r=r^0\}\cap V_1,
\] where $V_1$ is an open neighborhood of $p^\din_0$ such that
$\mathcal I_\epsilon$ can be parametrized as \beq{Ieps_para}
  \mathcal I_\epsilon
  &= \big\{
    ({\hat\beta},r,w_1,w_2,\xi,\kappa):
    r=r^0,\kappa=\epsilon\log(1/r^0),\\
    &\qquad (w_1,w_2,\xi)= (w_{1L},w_{2L},s)+ \alpha_1(-\beta_L,-v_L,1)
    + \epsilon \theta(a,\alpha_1,\epsilon),\\
    &\qquad |a|<\Delta_1,|\alpha_1|<\Delta_1
  \big\},
\] where the coordinates $({\hat\beta},r,w_1,w_2,\xi,\kappa)$
are defined in Proposition \ref{prop_pl}.
From \eqref{WsPL} we see that $\mathcal I_0$ and $W^s_0(\mathcal P_L)$
intersect transversally at $p^\din_0$,
and if we set \beq{def_LambdaL}
  \Lambda_L
  = \pi^s_{0,\mathcal P_L}
  \big(
    \mathcal I_0\cap W^s_0(\mathcal P_L)
  \big),
\] where $\pi^s_{0,\mathcal P_L}$ is the projection into $\mathcal P_L$
along stable fibers with respect to \eqref{fast_brk},
then \beq{LambdaL_para}
  \Lambda_L
  &= \big\{
    (\beta,r,w_1,w_2,\xi,\kappa):
    \beta=\rho_1,r=0,\kappa=0,\\
    &\qquad (w_1,w_2,\xi)= (w_{1L},w_{2L},s)+ \alpha_1(-\beta_L,-v_L,1),\\
    &\qquad |\alpha_1|<\Delta_1
  \big\}.
\] 

Similarly, by shrinking $r^0$ if necessary,
$\gamma_2$ intersects $\{r=r^0\}$ at a unique point \beq{def_pout0}
  p^\dout_0= \gamma_2\cap \{r=r^0\}.
\] Set \beq{def_Jeps}
  \mathcal J_\epsilon
  =W^s_\epsilon(\mathcal U_R)\cap \{r=r^0\}\cap V_2,
\] where $V_2$ is an open neighborhood of $p^\dout_0$
such that $\mathcal J_\epsilon$ has a parametrization analogous to \eqref{Ieps_para}.
Then $\mathcal J_0$ is transverse to $W^u_0(\mathcal P_R)$ at $p^\dout_0$,
and we set  \beq{def_LambdaR}
  \Lambda_R
  = \pi^u_{0,\mathcal P_R}
  \big(
    \mathcal J_0\cap W^u_0(\mathcal P_R)
  \big),
\] where $\pi^u_{0,\mathcal P_R}$ is the projection into $\mathcal P_R$
along unstable fibers with respect to \eqref{fast_brk}.

To connect $p^\din_0$ and $p^\dout$,
we have the following

\begin{prop}\label{prop_gamma0}
The system \eqref{fast_brk} has a trajectory \beq{def_gamma0}
  \gamma_0
  =\big\{
    (\beta,0,w_{1L},w_{20},s,\kappa_0):
    \beta\in (\rho_2,\rho_1)
  \big\},
\] which joins the points \[
  \pi^s_{\mathcal P_R}(p^\dout_0)\myflow{slow_PL}\tau_{10}\in \mathcal P_L
  \quad\text{and}\quad
  \pi^u_{\mathcal P_L}(p^\din_0)\myflow{slow_PR}(-\tau_{20})\in \mathcal P_R,
\] where \beq{def_kappa0}
  w_{20}
  = w_{2L}+ \frac{\rho_1}{\rho_1+\rho_2}(w_{2L}-w_{2R}),\quad
  \kappa_0
  = \frac{\rho_1(\rho_1-\rho_2)}{2\rho_2(\rho_1+\rho_2)},
\] and \beq{def_tau0}
  \tau_{10}
  =\frac{\rho_2}{\rho_1+\rho_2}(w_{2L}-w_{2R}),\quad
  \tau_{20}
  =\frac{\rho_1}{\rho_1+\rho_2}(w_{2L}-w_{2R}).
\] Moreover, if we set \beq{def_tLambda}
  {\widetilde\Lambda}_L
  = \Lambda_L\myflow{slow_PL} [\tau_{1-},\tau_{1+}]
  \quad\text{and}\quad
  {\widetilde\Lambda}_R
  = \Lambda_R\myflow{slow_PR} \big(-[\tau_{2-},\tau_{2_+}]\big),
\] where $\tau_{1-}<\tau_{10}<\tau_{1+}$ and $\tau_{2-}<\tau_{20}<\tau_{2+}$,
then $W^u_0({\widetilde\Lambda}_L)$ and $W^s_0({\widetilde\Lambda}_R)$
intersect transversally
along $\gamma_0$
in the space $\{r=0\}$.
\end{prop}

\begin{proof} Note that the restriction of the system \eqref{fast_brk}
on $\{r=0\}$ is simply $\dot\beta=B_1(\beta)$,
so every trajectory of \eqref{fast_brk} joins $\{\beta=\rho_1\}$ and $\{\beta=\rho_2\}$.
Also note that \[
  &\pi^s_{\mathcal P_R}(p^\dout_0)\myflow{slow_PL}\tau
  = (\rho_1,0,w_{1L},w_{2L},s,0)+ \tau (0,0,-1,0,B_2(\rho_1))\\
  &\pi^u_{\mathcal P_L}(p^\dout_0)\myflow{slow_PR}\tau
  =(\rho_2,0,w_{1R},w_{2R},s,0)
  + \tau (0,0,0,-1,0,B_2(\rho_2)),\; \forall\; \tau\in \mathbb R,
\] in $(\beta,r,w_1,w_2,\xi,\kappa)$-coordinates.
Hence $\gamma_0$ defined in \eqref{def_gamma0}
joins $\pi^s_{\mathcal P_R}(p^\dout_0)\myflow{slow_PL}\tau_{10}$
and $\pi^u_{\mathcal P_L}(p^\dout_0)\myflow{slow_PR}(-\tau_{20})$
if \beq{def_tau0_tmp}
  w_{20}= w_{2L}- \tau_{10}= w_{2R}+\tau_{20},\quad
  \kappa_0= B_2(\rho_1)\tau_{10}= -B_2(\rho_2)\tau_{20},
\] which gives \eqref{def_kappa0} and \eqref{def_tau0}.

Let $\widetilde\Lambda_L$ and $\widetilde\Lambda_R$
be defined in \eqref{def_tLambda}.
From the parameterizations
\eqref{WuPL} and \eqref{LambdaL_para}, we have \beq{Wu_tLambdaL}
  &W^u_0(\widetilde\Lambda_L)
  = \big\{
    (\beta,r,w_1,w_2,\xi,\kappa):
    r=0,\\
    &(w_1,w_2,\xi,\kappa)= (w_{1L},w_{2L},s,0)+ \alpha_1(-\beta_L,-v_L,1,0)
    + \tau_1 (0,-1,0,B_2(\rho_1)),\\
    &\beta\in (\rho_2,\rho_1),|\alpha_1|<\Delta_1, \tau_1\in [\tau_{1-},\tau_{1+}]
  \big\}
\] and \beq{Ws_tLambdaR}
  &W^s_0(\widetilde\Lambda_R)
  = \big\{
    (\beta,r,w_1,w_2,\xi,\kappa):
    r=0,\\
    &(w_1,w_2,\xi,\kappa)= (w_{1R},w_{2R},s,0)+ \alpha_1(-\beta_R,-v_R,1,0)
    - \tau_2 (0,-1,0,B_2(\rho_2)),\\
    &\beta\in (\rho_2,\rho_1),|\alpha_2|<\Delta_1, \tau_2\in [\tau_{2-},\tau_{2+}
  \big\}.
\]
Fix any $q_0\in \gamma_0$, we have \[
  T_{q_0}W^u_0({\widetilde\Lambda}_L)
  = \mathrm{Span}\left\{
    \p 1\\ 0\\ 0\\ 0\\ 0\\ 0\pp,
    \p 0\\ 0\\ -\beta_L\\ -v_L\\ 1\\ 0\pp,
    \p 0\\ 0\\ 0\\ 1\\ 0\\ B_2(\rho_1)\pp
  \right\}
\] and \[
  T_{q_0}W^s_0({\widetilde\Lambda}_R)
  = \mathrm{Span}\left\{
    \p 1\\ 0\\ 0\\ 0\\ 0\\ 0\pp,
    \p 0\\ 0\\ -\beta_R\\ -v_R\\ 1\\ 0\pp,
    \p 0\\ 0\\ 0\\ 1\\ 0\\ B_2(\rho_2)\pp
  \right\}.
\] Hence $T_{q_0}W^u_0({\widetilde\Lambda}_L)$
and $T_{q_0}W^s_0({\widetilde\Lambda}_R)$
span the space $\{r=0\}$.
This means $W^u_0({\widetilde\Lambda}_L)$ and $W^s_0({\widetilde\Lambda}_R)$
intersect transversally in the space $\{r=0\}$ at $q_0$.
\end{proof}

Let $\gamma_0$ be the trajectory defined in Proposition \ref{prop_gamma0}.
We set \beq{def_q0}
  q_0= \gamma_0\cap \Gamma,
\] where  \beq{def_Gamma}
  \Gamma= \{(\beta,r,w_1,w_2,\xi,\kappa): \beta=\frac{\rho_1+\rho_2}2\}.
\] Then \beq{link_pq0}
  \pi^u_{\mathcal P_L}(q_0)
  =\pi^s_{\mathcal P_L}(p^\din_0)
  \myflow{slow_PL} \tau_{10},\quad
  \pi^s_{\mathcal P_R}(q_0)
  =\pi^u_{\mathcal P_R}(p^\dout_0)
  \myflow{slow_PR} (-\tau_{20}).
\] Let \beq{def_sigma1}
  \sigma_1
  &= \pi^s_{\mathcal P_L}(p^\din_0)\myflow{slow_PL} [0,\tau_{10}]\\
  &=\{(\rho_2,0,w_{1L},w_{2L}-\tau,s,B_2(\rho_1)\tau): \tau\in [0,\tau_{10}]\}
\] and \beq{def_sigma2}
  \sigma_2
  &= \pi^u_{\mathcal P_L}(p^\din_0)\myflow{slow_PR} [-\tau_{20},0]\\
  &=\{(\rho_1,0,w_{1R},w_{2R}+\tau,s,-B_2(\rho_2)\tau): \tau\in [0,\tau_{20}]\}
\] in $(\beta,r,w_1,w_2,\xi,\kappa)$-coordinates.
Then we obtain the singular configuration \beq{def_singular_config}
  \gamma_1\cup \sigma_1\cup \gamma_0
  \cup \sigma_2\cup \gamma_2
\] connecting $\mathcal U_L$ and $\mathcal U_R$.
See Fig \ref{fig_config}.

\begin{figure}\centering
\boxed{
\includegraphics[trim = 2cm 7.4cm 1.6cm 7.45cm, clip, width=.68\textwidth]{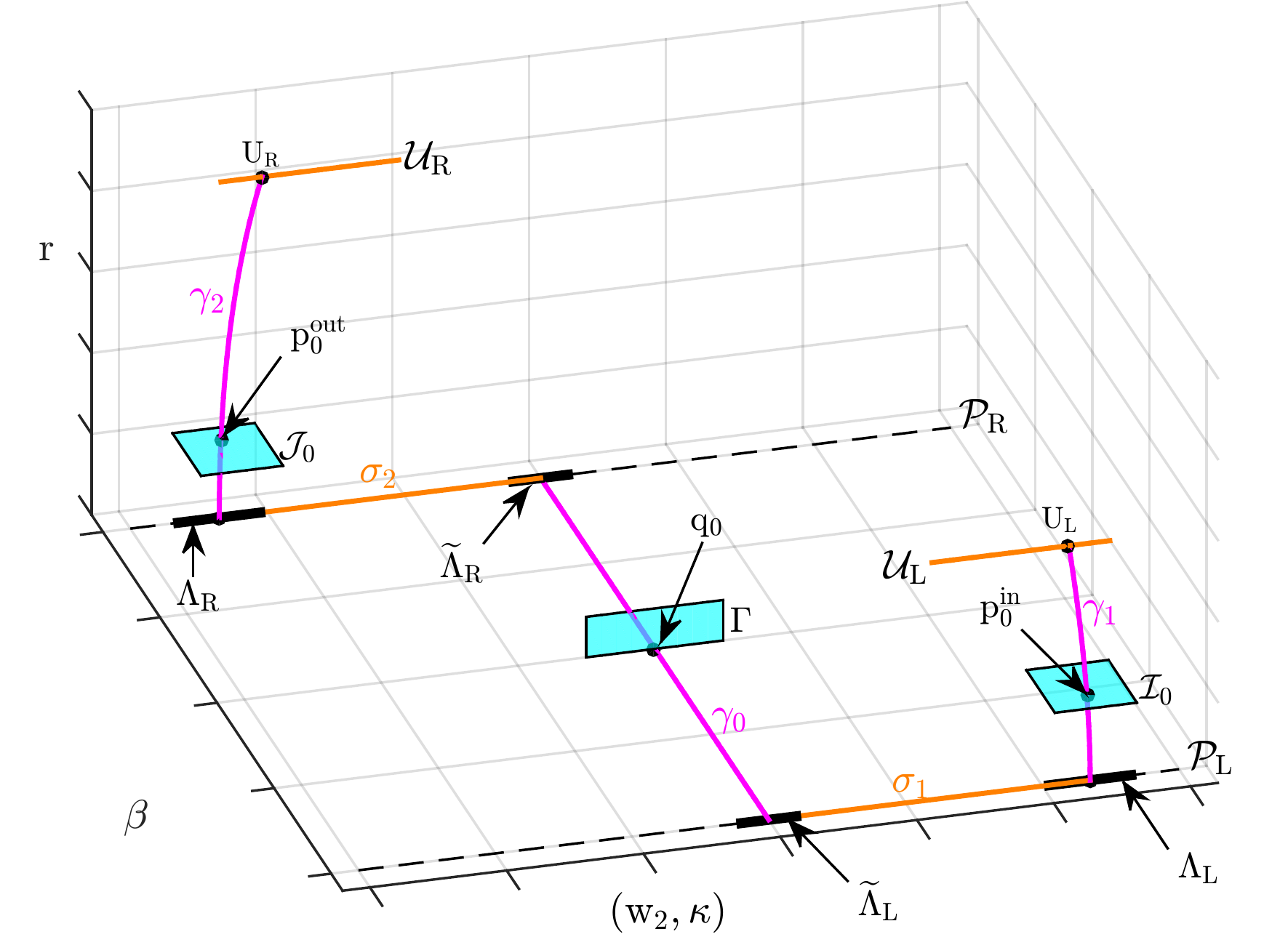}
}
\caption{The singular configuration
$\gamma_1\cup \sigma_1\cup \gamma_0\cup \sigma_2\cup \gamma_2$
connects $\mathcal U_L$ and $\mathcal U_R$.}
\label{fig_config}
\end{figure}

\section{Completing the Proof of the Main Theorem} 
\label{sec_complete}
We split the proof of the main theorem into two parts.
In the first subsection
we prove the existence of solutions of
the boundary value problem \eqref{dafermos_similar}, \eqref{bc_u_infty},
and show that \eqref{est_v_max} holds.
In the second subsection
we derive the weak limit \eqref{weak_limit}.

\subsection{Existence of Trajectories}
\label{subsec_existence}
\begin{prop}\label{prop_in_and_out}
Assume $\mathrm{(H1)}$-$\mathrm{(H3)}$.
Let $p^\din_0$, $p^\dout_0$, $q_0$,
$\mathcal I_\epsilon$, $\mathcal J_\epsilon$ and $\Sigma$
be defined in Section \ref{subsec_transversal}.
Then for each small $\epsilon>0$,
there exist $p^\din_\epsilon\in\mathcal I_\epsilon$,
$p^\dout_\epsilon\in \mathcal J_\epsilon$,
$q_\epsilon\in \Gamma$
and $T_{1\epsilon},T_{2\epsilon}>0$ such that \beq{pq_link}
  p^\din_\epsilon= q_\epsilon \cdot (-T_{1\epsilon}),\quad
  p^\dout_\epsilon= q_\epsilon \cdot T_{2\epsilon},
\] where $\cdot$ denotes the flow of \eqref{sf_brk}, satisfying \beq{limit_p0}
  (p^\din_\epsilon,p^\dout_\epsilon,q_\epsilon)
  = (p^\din_0,p^\dout_0,q_0)+ o(1)
\] and \beq{est_Teps}
  T_{1\epsilon}= \big(
    \tau_{10}+ o(1)
  \big)\epsilon^{-1},\quad
  T_{2\epsilon}= \big(
    \tau_{20}+ o(1)
  \big)\epsilon^{-1},
\] as $\epsilon\to 0$,
where $\tau_{10}$ and $\tau_{20}$ are defined in \eqref{def_tau0}.
Moreover, if we set $\kappa_\epsilon(\sigma)$
to be the $\kappa$-coordinate of $q_\epsilon\cdot\sigma$,
$\sigma\in [-T_{1\epsilon},T_{2\epsilon}]$,
then \beq{est_kappa}
  \max_{\sigma\in [-T_{1\epsilon},T_{2\epsilon}]}\kappa_\epsilon(\sigma)
  = \kappa_0+ o(1),
\] where $\kappa_0$ is defined in \eqref{def_kappa0}.
\end{prop}

\begin{proof}
We will apply the Exchange Lemma (Theorem \ref{thm_EL}) with $(k,m,l,\sigma)=(1,1,3,1)$.
From \eqref{Ieps_para} and \eqref{WsPL},
we know that
the $(k+\sigma)$-manifold $\mathcal I_0$ is transverse to 
the $(m+l)$-manifold $W^s_0(\mathcal P_L)$ at $p^\din_0$,
and the image of the projection \[
  \pi^s_{\mathcal P_L}\big(
    \mathcal I_0\cap W^s_0(\mathcal P_L)
  \big)
  = \Lambda_L
\] is $\sigma$-dimensional,
so $\mathrm{(T1)}$ in the Exchange Lemma holds.
The limiting slow system on $\mathcal P_L$ is governed by \eqref{slow_PL},
and by the parametrization \eqref{LambdaL_para} of $\Lambda_L$
it follows that $\mathrm{(T2)}$ holds.
Also it is clear that $\mathrm{(T3)}$ holds with $\tau_0=\tau_{10}$,
where $\tau_{10}$ is defined in \eqref{def_tau0}.
Theorem \ref{thm_EL} implies that
there exists a neighborhood $V_0$ of $q_0$ such that \beq{closeness_PL}
  \text{
    $\mathcal I_\epsilon^*\cap V_0$
    is $C^1$ $O(\epsilon)$-close to $W^u_0({\widetilde\Lambda}_L)\cap V_0$,
  }
\] where $\mathcal I_\epsilon^*= \mathcal I_\epsilon\cdot [0,\infty)$.
Similarly, \beq{closeness_PR}
  \text{
    $\mathcal J_\epsilon^*\cap V_0$
    is $C^1$ $O(\epsilon)$-close to $W^s_0({\widetilde\Lambda}_R)\cap V_0$,
  }
\]
where $\mathcal J_\epsilon^*= \mathcal J_\epsilon\cdot (-\infty,0]$.
From Proposition \ref{prop_gamma0},
it follows that the projections of $\mathcal I_\epsilon^*$ and $\mathcal J_\epsilon^*$
in the $5$-dimensional space $\{r=0\}$
intersect transversally at a unique point in $\Gamma$ near $q_0$.
For the relation $r=\exp(-\kappa/\epsilon)$,
we then recover a unique intersection point \[
  q_\epsilon
  \in \mathcal I_\epsilon^*\cap \mathcal J_\epsilon^*\cap \Gamma
\] in $(\beta,r,w_1,w_2,\xi,\kappa)$-space.
By construction we have \eqref{pq_link} and \eqref{limit_p0}.
The estimates \eqref{est_Teps}
follows from \eqref{est_Teps_EL}.
Note that \[
  \max_{\sigma_1\cup \gamma_0\cup \sigma_2}\kappa= \kappa_0,
\] where $\sigma_1$, $\sigma_2$ and $\gamma_0$
are defined in Section \ref{subsec_transversal},
so we obtain \eqref{est_kappa}.
\end{proof}

\begin{prop}\label{prop_existence}
Assume $\mathrm{(H1)}$-$\mathrm{(H3)}$ hold.
Let $q_\epsilon=(\beta_\epsilon^0,r_\epsilon^0,
w_{1_\epsilon}^0,w_{2_\epsilon}^0,\xi_\epsilon^0,\kappa_\epsilon^0)
\in\Gamma$
be defined in Proposition \ref{prop_in_and_out}.
Let $v_\epsilon^0=\exp(\kappa_\epsilon^0/\epsilon)$ and \beq{def_tilde_beta}
  (\tilde{\beta}_\epsilon,\tilde{v}_\epsilon,
  \tilde{w_1}_\epsilon,\tilde{w_2}_\epsilon)(\xi)
  = (\beta_\epsilon^0,v_\epsilon^0,
  w_{1\epsilon}^0,w_{2\epsilon}^0)
  \myflow{dafermos_similar_xi}
  (\xi-\xi_\epsilon^0),
\] or equivalently, \beq{def_tilde_beta2}
  (\tilde{\beta}_\epsilon,\tilde{v}_\epsilon,
  \tilde{w_1}_\epsilon,\tilde{w_2}_\epsilon,\xi)
  = (\beta_\epsilon^0,v_\epsilon^0,
  w_{1\epsilon}^0,w_{2\epsilon}^0,\xi_\epsilon^0)
  \myflow{sf_bv} \left(\frac{\xi-\xi_\epsilon^0}\epsilon\right).
\] Then $(\tilde{\beta}_\epsilon,\tilde{v}_\epsilon)$
is a solution of \eqref{dafermos_similar} and \eqref{bc_u_infty},
and it satisfies \eqref{est_v_max}.
\end{prop}

\begin{proof}
Since \eqref{dafermos_similar} and \eqref{dafermos_similar_xi} are equivalent,
and $(\tilde{\beta}_\epsilon,\tilde{v}_\epsilon,
\tilde{w_1}_\epsilon,\tilde{w_2}_\epsilon)(\xi)$
is a solution of \eqref{dafermos_similar_xi},
we know $(\tilde{\beta}_\epsilon,\tilde{v}_\epsilon)$
is a solution of \eqref{dafermos_similar}.

Let $T_{1\epsilon},T_{2\epsilon}\in\mathbb R$
be defined in Proposition \ref{prop_in_and_out}. Then \[
  q_\epsilon
  \underset{\eqref{sf_brk}}{\bullet} (-T_{1\epsilon}/\epsilon)\in \mathcal I_\epsilon,
  \quad
  q_\epsilon
  \underset{\eqref{sf_brk}}{\bullet} (T_{2\epsilon}/\epsilon)\in \mathcal J_\epsilon.
\] Since $\mathcal I_\epsilon\subset W_\epsilon^u(\mathcal U_L)$
and $\mathcal J_\epsilon\subset W_\epsilon^s(\mathcal U_R)$,
from \eqref{est_dist_ul} and \eqref{est_dist_ur}
we have \[
  \lim_{t\to -\infty}\mathrm{dist}(q_\epsilon\myflow{sf_brk}t,\mathcal U_L)=0,
  \quad
  \lim_{t\to \infty}\mathrm{dist}(q_\epsilon\myflow{sf_brk}t,\mathcal U_R)=0,
\] which implies \eqref{bc_u_infty}.
Since $\kappa_\epsilon=\epsilon\log(v_\epsilon)$,
from \eqref{est_kappa} we obtain \eqref{est_v_max}.
\end{proof}

\subsection{Convergence of Trajectories}
\label{subsec_convergence}
Based on the results in Proposition \ref{prop_in_and_out},
we first derive some estimates for the self-similar solution $\tilde u_\epsilon(\xi)$.

\begin{prop}\label{prop_convergence}
Let $\tilde u_\epsilon=(\tilde\beta_\epsilon,\tilde v_\epsilon)$
be the solution of \eqref{dafermos_similar} and \eqref{bc_u_infty}
in Proposition \ref{prop_existence}.
Let $p_\epsilon^\din$ and $p_\epsilon^\dout$
be defined in Proposition \ref{prop_in_and_out}.
Then \begin{align}
  &|\xi_\epsilon^\din-s|+|\xi_\epsilon^\dout-s|= o(1)
  \label{est_xi}\\
  &\int_{-\infty}^{\xi_\epsilon^\din}|\tilde{u}(\xi)-u_L|\;d\xi
  +\int_{\xi_\epsilon^\dout}^{\infty}|\tilde{u}(\xi)-u_R|\;d\xi = o(1)
  \label{est_int_outer}\\
  &\int_{\xi_\epsilon^\din}^{\xi_\epsilon^\dout}\tilde{u}(\xi)\;d\xi= (0,e_0)+ o(1)
  \label{est_v_inner}
\end{align}
as $\epsilon\to0$,
where $\xi_\epsilon^\din$ and $\xi_\epsilon^\dout$
are $\xi$-coordinates of $p_\epsilon^\din$ and $p_\epsilon^\dout$, respectively.
\end{prop}

\begin{proof}
Note that $s$ is the $\xi$-coordinate of $p_0^\din$, so \[
  |\xi_\epsilon^\din-s|
  \le |p_\epsilon^\din-p_0^\din|,
\] which tends to zero by \eqref{limit_p0}.
Similarly, $|\xi_\epsilon^\dout-s|\to 0$.
This gives \eqref{est_xi}.

Since every point in $\mathcal U_L$ has $u$-coordinate equal to $u_L$, \[
  |\tilde{u}(\xi)-u_L|
  \le \mathrm{dist}\big((\tilde u(\xi),\tilde w(\xi),\xi),\mathcal U_L\big)
  = \mathrm{dist}\big(
    (u_\epsilon^0,w_\epsilon^0,\xi_\epsilon^0)
    \myflow{sf_u}\frac{\xi-\xi_\epsilon^0}{\epsilon},
    \mathcal U_L
  \big),
\] where the last equality follows from \eqref{def_tilde_beta2}.
Using \eqref{est_dist_ul},
the last term is $\le C\exp\big(\nu\frac{\xi-\xi_\epsilon^0}{\epsilon}\big)$.
Since $\xi_\epsilon^\din<\xi_\epsilon^0$, it follows that \[
  \int_{-\infty}^{\xi_\epsilon^\din}|\tilde{u}(\xi)-u_L|\;d\xi
  &\le \int_{-\infty}^{\xi_\epsilon^\din}
    C\exp\big(\nu\frac{\xi-\xi_\epsilon^0}{\epsilon}\big)
  \;d\xi
  &\le \int_{-\infty}^{\xi_\epsilon^0}
    C\exp\big(\nu\frac{\xi-\xi_\epsilon^0}{\epsilon}\big)
  \;d\xi
  =\frac{\epsilon}\nu C,
\] A similar inequality holds for
$\int_{\xi_\epsilon^\dout}^{\infty}|\tilde{u}(\xi)-u_R|\;d\xi$,
so we obtain \eqref{est_int_outer}.

Since $\tilde\beta_\epsilon(\xi)$ is uniformly bounded in $\epsilon$,
its integral between $\xi_\epsilon^\din$ and $\xi_\epsilon^\dout$
is $o(1)$ by \eqref{est_xi},
and this proves the first part of \eqref{est_v_inner}.
From the equation of $\dot\xi$ in \eqref{sf_brk},
denoting the time variable by $\zeta$,
we can write $\xi=\xi(\zeta)$ by \beq{xi_de}
  \xi(0)=\xi_\epsilon^0,\quad
  \frac{d\xi}{d\zeta}=\epsilon \tilde r_\epsilon(\xi),
\] where $\tilde r_\epsilon(\xi)= 1/\tilde v_\epsilon(\xi)$.
From \eqref{pq_link} we have \beq{xi_bc}
  \xi(-T_{1\epsilon})=\xi_\epsilon^\din,\quad
  \xi(T_{2\epsilon})=\xi_\epsilon^\dout.
\] From \eqref{xi_de} and \eqref{xi_bc} it follows that \[
  \int_{\xi_\epsilon^\din}^{\xi_\epsilon^\dout}\tilde v(\xi)\;d\xi =
  \int_{\xi_\epsilon^\din}^{\xi_\epsilon^\dout} \frac1{\tilde{r}(\xi)}\;d\xi
  =\int_{-T_{1\epsilon}}^{T_{2\epsilon}} 
    \epsilon
  \;d\zeta
  = \epsilon(T_{1\epsilon}+ T_{2\epsilon}),
\] which converges to $w_{2L}-w_{2R}= e_0$ by \eqref{est_Teps}.
This proves \eqref{est_v_inner}.
\end{proof}

From the estimates in Proposition \ref{prop_convergence},
we can derive the weak convergence of $\tilde u(\xi)$ as follows.

\begin{prop}\label{prop_weak_limit_bv_xi}
Let $\tilde{u}_\epsilon=(\tilde\beta_\epsilon,\tilde v_\epsilon)$
be the solution of \eqref{dafermos_similar} and \eqref{bc_u_infty}
given in Proposition \ref{prop_existence}.
Then \beq{weak_limit_tilde_u}
  \tilde{u}_\epsilon\rightharpoonup u_L+ (u_R-u_L)\mathrm{H}(\xi-s)
  + (0,e_0)\delta_0(\xi-s)
\] in the sense of distributions as $\epsilon\to 0$.
\end{prop}

\begin{proof}
Let $\psi\in C^\infty_c(\mathbb R)$ be a smooth function with compact support.
From \eqref{est_int_outer} we have \[
  \left|\int_{-\infty}^{\xi_\epsilon^\din} \psi(\xi)(\tilde{u}(\xi)-u_L)\;d\xi\right|
  \le \|\psi\|_{L^\infty}\int_{-\infty}^{\xi_\epsilon^\din}|\tilde{u}(\xi)-u_L|\;d\xi
  \le \|\psi\|_{L^\infty}C\epsilon,
\] which implies \[
  \int_{-\infty}^{\xi_\epsilon^\din}\psi(\xi)\tilde{u}(\xi)\;d\xi
  =\left(\int_{-\infty}^{\xi_\epsilon^\din}\psi(\xi)\;d\xi\right) u_L+ o(1)
  =\left(\int_{-\infty}^s\psi(\xi)\;d\xi\right) u_L+ o(1).
\] A similar inequality holds for $\int_{\xi_\epsilon^\dout}^\infty\psi u\;d\xi$,
so \beq{est_psi_outer}
  \int_{\mathbb R\setminus [\xi_\epsilon^\din,\xi_\epsilon^\dout]}
  \psi(\xi)\tilde{u}(\xi)\;d\xi
  =\left(\int_{-\infty}^s\psi(\xi)\;d\xi\right) u_L
  + \left(\int_s^{\infty}\psi(\xi)\;d\xi\right) u_R
  + o(1).
\] On the other hand, 
from \eqref{est_xi} and \eqref{est_v_inner} we have \[
    \int_{\xi_\epsilon^\din}^{\xi_\epsilon^\dout}
    \big|(\psi(\xi)-\psi(s)) \tilde{u}_\epsilon(\xi)\big|\;d\xi
    &\le \left( \max_{\xi\in [\xi_\epsilon^\din,\xi_\epsilon^\dout]}|\psi(\xi)-\psi(s)| \right)
    \left( \int_{\xi_\epsilon^\din}^{\xi_\epsilon^\dout} \tilde{u}_\epsilon(\xi)\;d\xi \right)\\
    &= o(1)\Big((0,e_0)+o(1)\Big)
    = o(1)
\] which implies \beq{est_psi_v_inner}
  \int_{\xi_\epsilon^\din}^{\xi_\epsilon^\dout}\psi(\xi) \tilde{u}_\epsilon(\xi)\;d\xi
  = \psi(s)\int_{\xi_\epsilon^\din}^{\xi_\epsilon^\dout} \tilde{u}_\epsilon(\xi)\;d\xi
  + o(1)
  =\psi(s)(0,e_0)+ o(1).
\] Combining \eqref{est_psi_outer} and \eqref{est_psi_v_inner}
we obtain \beq{est_psi_v}
  \int_{-\infty}^\infty \psi(\xi)\tilde{u}_\epsilon(\xi)\;d\xi
  = u_L\int_{-\infty}^{s}\psi(\xi)\;d\xi
  +u_R\int_{s}^\infty \psi(\xi)\,d\xi
  + (0,e_0)\psi(s)
  + o(1).
\]
This holds for all $\psi$, so we have \eqref{weak_limit_tilde_u}.
\end{proof}

Converting the results of Proposition \ref{prop_weak_limit_bv_xi}
from self-similar variables
to physical space variables,
we obtain the following

\begin{prop}\label{prop_weak_limit_bv}
Let $\tilde{u}_\epsilon=(\tilde\beta_\epsilon,\tilde v_\epsilon)$
be the solution of \eqref{dafermos_similar} and \eqref{bc_u_infty}
given in Proposition \ref{prop_existence}.
Let $u_\epsilon(x,t)=\tilde u_\epsilon(x/t)$.
Then the weak convergence \eqref{weak_limit_beta} and \eqref{weak_limit_v} holds.
\end{prop}

\begin{proof}
Let $\varphi\in C^\infty_c(\mathbb R\times \mathbb R_+)$.
From \eqref{weak_limit_tilde_u} we have  \[
  &\int_0^\infty\int_{-\infty}^\infty \varphi(x,t)u_\epsilon(x,t)\,dx\,dt
  =\int_0^\infty t\int_{-\infty}^\infty \varphi(t\xi,t)\tilde{u}_\epsilon(\xi)\,d\xi\,dt\\
  &=\int_0^\infty t\left\{
      u_L\int_{-\infty}^s \varphi(t\xi,t)\,d\xi
      + (0,e_0)\varphi(st,t)
      + u_R\int_s^\infty \varphi(t\xi,t)\,d\xi
  \right\}dt+ o(1)\\
  &= u_L\int_0^\infty\!\! \int_{-\infty}^{st}\varphi(x,t)\,dx\,dt
  +u_R\int_0^\infty\!\! \int_{st}^{\infty}\varphi(x,t)\,dx\,dt
  +(0,e_0)\int_0^\infty t\varphi(st,t)\,dt
  +o(1).
\] From \eqref{def_delta}, this means \eqref{weak_limit} holds.
\end{proof}

Now Propositions \ref{prop_existence} and \ref{prop_weak_limit_bv}
complete the proof of the Main Theorem.

\section{Numerical Simulations} \label{sec_simulation}
Some numerical solutions for (\ref{claw_bv})
using the Lax-Friedrichs scheme
are shown in Figure \ref{fig_LF}.
The solutions appear to grow unboundedly as the number of steps increases.

\begin{figure}[t]
\centering
\includegraphics[trim = 1cm 6.8cm 2cm 7cm, clip, width=.49\textwidth]{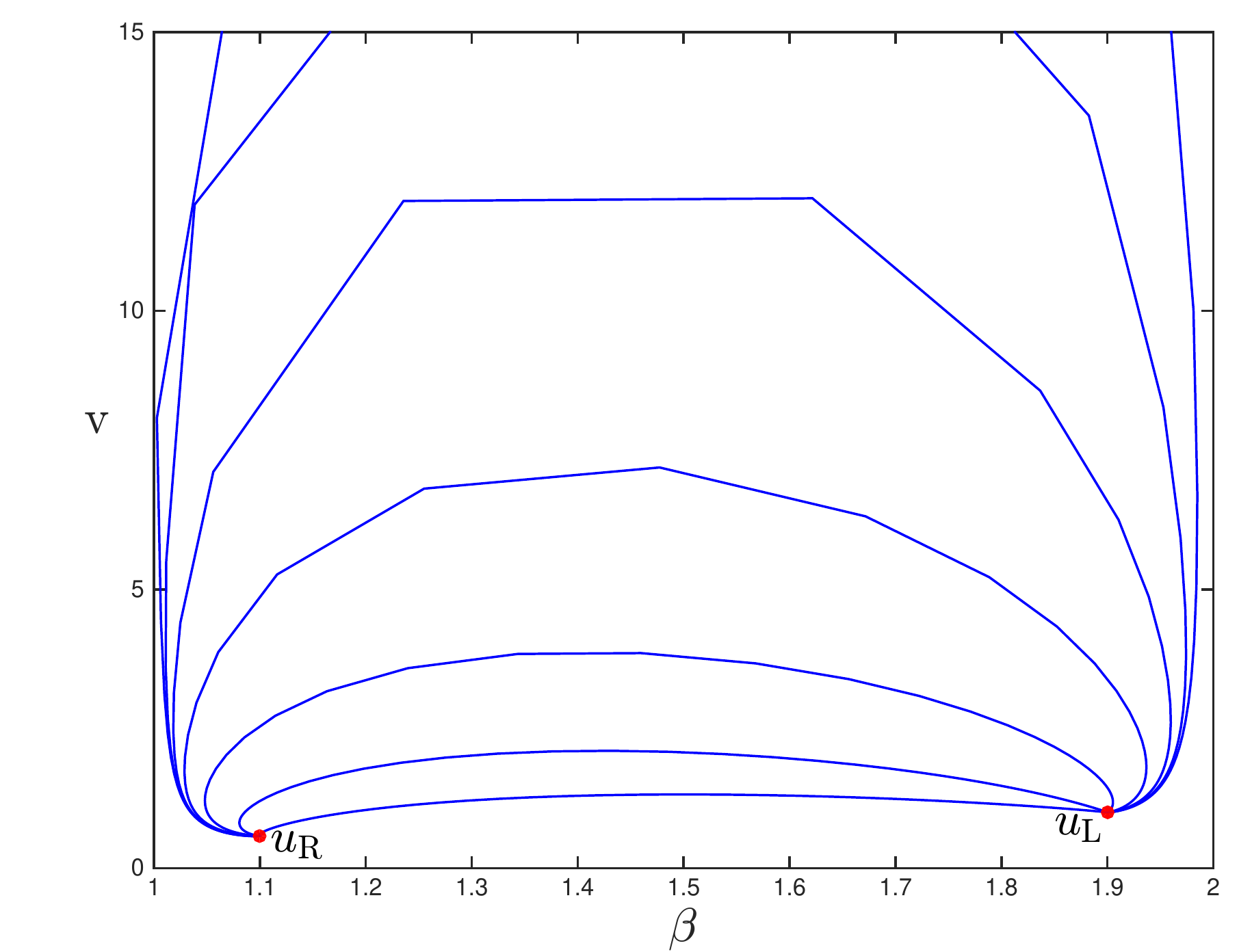}
\includegraphics[trim = 1cm 6.8cm 2cm 7cm, clip, width=.49\textwidth]{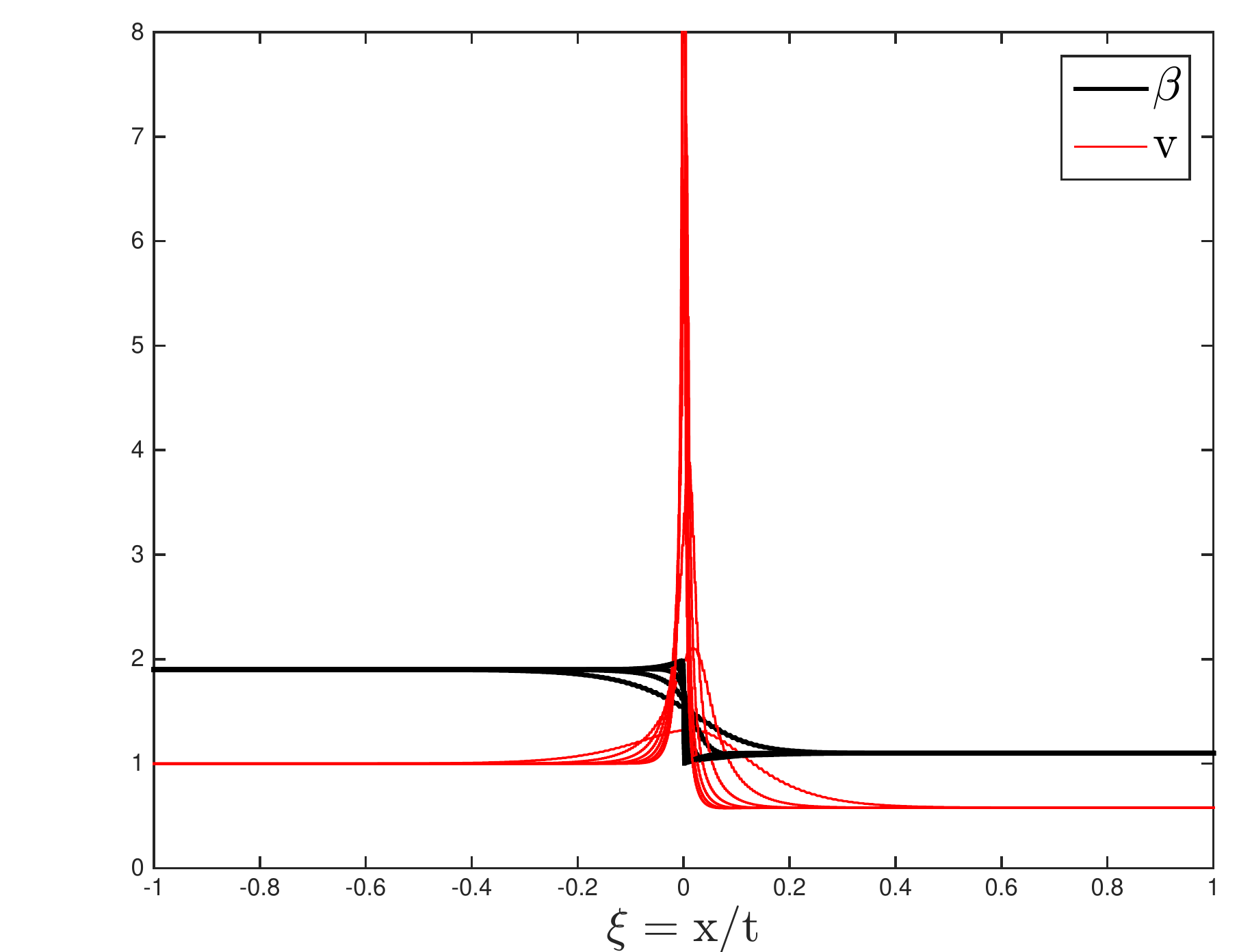}
\caption{Lax-Friedrichs Scheme up to $20,000$ steps with CFL$=0.05$}
\label{fig_LF}
\end{figure}

Also some numerical approximations for (\ref{dafermos_u})
are shown in Fig \ref{fig_shooting}.
The algorithm was a shooting method following the descriptions in \cite{Keyfitz:2003}.
Since $w_1$ and $\xi$ are essentially constant near the shock,
we project the trajectories in the $(\beta,r,w_2)$ space.
Note that $w_2(\xi)$ does not converge as $\xi\to \pm\infty$
while $x_2=w_2+(\xi-s) v$ converges,
we replace $w_2$ by $x_2$ (again, following \cite{Keyfitz:2003}).
Note that $x_2$ is a mild modification of $w_2$ near the shock
since within the $\epsilon$-neighborhood of $\xi=s$
the difference between $x_2$ and $w_2$
is of order $o(1)$.

As $\epsilon$ decreases,
the minimal value of $r$-coordinate on the trajectories
in Fig \ref{fig_shooting} tends to zero.
This means the maximum of $v$ tends to infinity.
Also observe that the change of the value of $x_2$
concentrates in the vicinity of $\beta=\rho_1$ and $\beta=\rho_2$.
This is consistent with our proof for the main theorem.

\begin{figure}[t]
\centering
{
\includegraphics[trim = 1cm 7.4cm 2cm 7.4cm, clip, width=.64\textwidth]{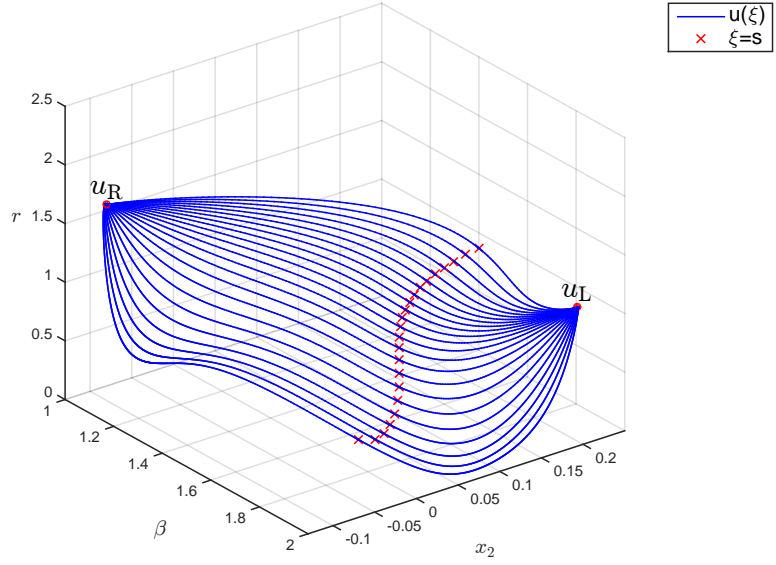}
}
\caption{Trajectories for $\epsilon u''=(Df(u)-s)u'$
as $\epsilon$ decreases from $1$ to $0.01$.
The variable $x_2$ is a modification of $w_2$.}
\label{fig_shooting}
\end{figure}

\section{Proof of Proposition \ref{prop_h3}} \label{sec_h3}

If the section we prove the sufficiency
of the conditions in Proposition \ref{prop_h3} for $\mathrm{(H3)}$,
which says that there is
a trajectory for \eqref{fast_bv}
connecting $(u_L,w_L,s)$ and $\{\beta=\rho_1,v=+\infty\}$,
and a trajectory connecting $(u_R,w_R,s)$ and $\{\beta=\rho_2,v=+\infty\}$.
We will focus on finding the first trajectory
while finding the second one is similar.

We will switch back and forth between $(\beta,v)$- and $(\beta,r)$-coordinates,
where $r=1/v$.
The system \eqref{fast_bv} is converted to \eqref{fast_brk} in $(\beta,r)$-coordinates.
It suffices to find trajectories connecting
$u_L=(\beta_L,r_L)$ and $p_L\equiv (\rho_1,0)$
for \eqref{fast_brk} with $(w,\xi)=(w_L,s)$.
From $\mathrm{(H1)}$ we know that $u_L$ is a source,
and we will also see that $p_L$ is a saddle.
Our strategy is to construct a negatively invariant region
in which every trajectory goes backward to $u_L$,
and one of those trajectories goes forward to $p_L$.
See Fig \ref{fig_h3_trapping}.

To construct a such region,
we first study the flow on the boundary
of the feasible region $\{\rho_2\le \beta\le \rho_1\}$.
The equation of $\dot\beta$ in \eqref{fast_bv} with $(w,\xi)=(w_L,s)$ is \[
  \dot{\beta}= -s\rho_1-w_{1L}\quad
  \text{ on}\;\; \{\beta=\rho_1\}.
\] Hence the proposition below implies that
the region $\{\beta\le \rho_1\}$ is negatively invariant
Similarly, for \eqref{fast_bv} with $(w,\xi)=(w_R,s)$,
the region $\{\beta\ge \rho_2\}$
is positively invariant.

\begin{lem}\label{lem_bdry_sign}
If $\mathrm{(H1)}$ holds,
then \beq{est_bdry}
  s\rho_1+w_{1L}<0
  \quad\text{and}\quad
  s\rho_2+w_{1R}<0,
\] where $s$, $w_{1L}$ and $w_{1R}$
are as defined in \eqref{def_s} and \eqref{def_w}.
\end{lem}

\begin{proof}
By definition of $s$ and $w_{1L}$ we have \[
  s\rho_1+w_{1L}
  &=s\rho_1
  +v_LB_1(\beta_L)-s\beta_L\\
  &=\frac{v_LB_1(\beta_L)-v_RB_1(\beta_R)}{\beta_L-\beta_R}(\rho_1-\beta_L)
  +v_LB_1(\beta_L).
\] From Proposition \ref{prop_overcompressive},
we know $\mathrm{(H1)}$ implies $\beta_R<\beta_L$ and \[
  v_R
  \le v_L\left(\frac{B_1(\beta_L)-2B_2(\beta_L)(\beta_L-\beta_R)}{B_1(\beta_R)}\right).
\] Since $B_1(\beta_L)<0$, it follows that \[
  s\rho_1+w_{1L}
  &\le \frac{\rho_1-\beta_L}{\beta_L-\beta_R}
  v_L\big(2B_2(\beta_L)(\beta_L-\beta_R)\big)
  + v_L B_1(\beta_L)\\
  &= \left(\frac{(\rho_1-\beta_L)(\beta_L^2-\rho_1\rho_2)}{\beta_L^2}
    +\frac{(\rho_1-\beta_L)(\rho_2-\beta_L)}{\beta_L}
  \right)v_L\\
  &=\frac{-\rho_2(\rho_1-\beta_L)^2v_L}{\beta_L^2}<0.
\]
Similarly, using $s\rho_2+w_{1R}=s\rho_2+v_RB_1(\beta_R)-s\beta_R$,
one obtains $s\rho_2+w_{1R}<0$.
\end{proof}

\begin{figure}[t]\centering
\fbox{
\includegraphics[trim = 3.4cm 8.2cm 2.4cm 7.6cm, clip, width=.44\textwidth]{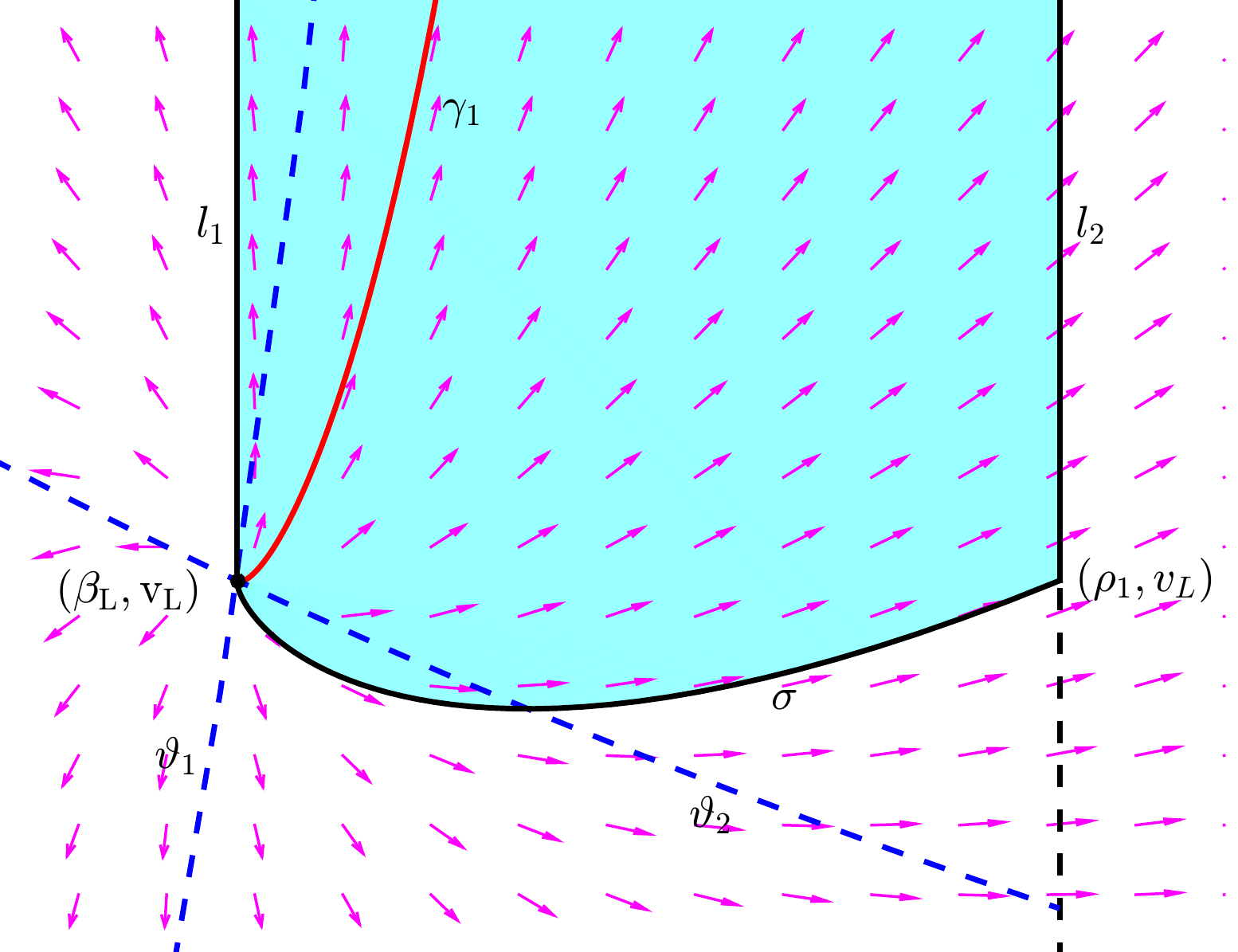}
}
\fbox{
\includegraphics[trim = 3.4cm 8.2cm 2.4cm 7.6cm, clip, width=.44\textwidth]{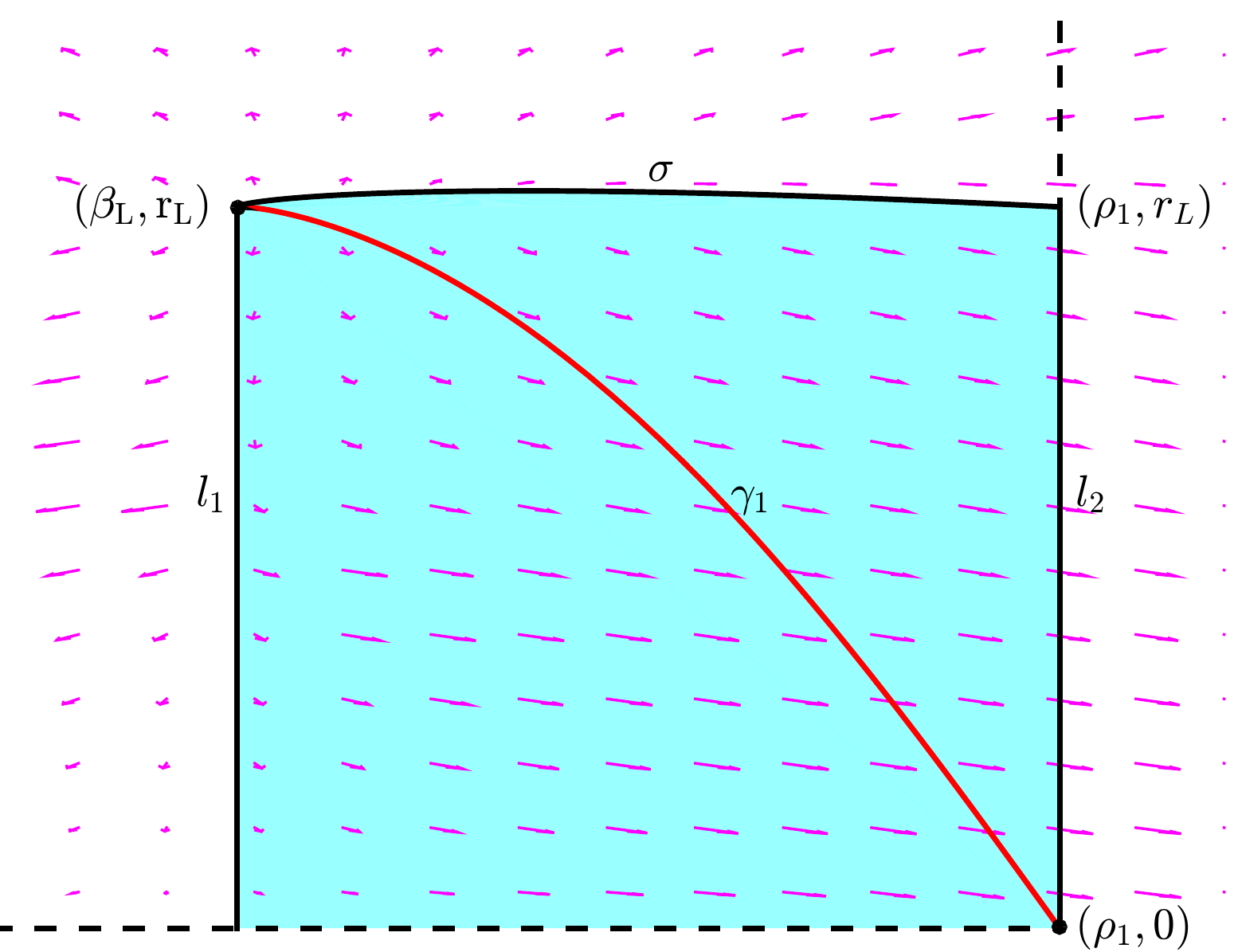}
}
\caption{
Phase portraits for \eqref{deq_fast_L} in $(\beta,v)$ space
and \eqref{deq_fast_L_r} in $(\beta,r)$ space.
The shaded region $V$ is a backward invariant region in which
every backward trajectory tends to $u_L$,
and $\gamma_1$ is the unique trajectory in $V$
which tends to $p_L=(\rho_1,0)$.}
\label{fig_h3_trapping}
\end{figure}

\newtheorem*{prop_h3_repeat}{Proposition \ref{prop_h3}}
\begin{prop_h3_repeat}
Suppose $\mathrm{(H1)}$ holds.
If $\beta_R<\sqrt{\rho_1\rho_2}<\beta_L$,
$w_{10}<0$, $w_{2R}<0<w_{2L}$,
and $|s|$ is sufficiently small,
then $\mathrm{(H3)}$ holds.
\end{prop_h3_repeat}

\begin{proof}
We focus on $u_L$ while the proof for $u_R$ is similar.
As mentioned at the beginning of this section and Fig \ref{fig_h3_trapping},
we first construct a negatively invariant region
in which every trajectory goes backward to $u_L$,
and then show that one of those trajectories goes forward to $p_L$.

Consider \eqref{fast_bv} with $(w,\xi)=(w_L,s)$.
That is, \beq{deq_fast_L}
  &\dot\beta= vB_1(\beta)-s\beta-w_{1L}\\
  &\dot v=v^2B_2(\beta)-sv-w_{2L}.
\]
The null-clines for this system are \begin{align}
  &\dot\beta=0:\quad
  v=\vartheta_1(\beta)
  :=\frac{\beta(s\beta+w_{1L})}{(\beta-\rho_1)(\beta-\rho_2)}\\
  &\dot{v}=0:\quad
  v=\vartheta_2(\beta)
  :=\frac{s\beta+\sqrt{s^2\beta^2+2w_{2L}(\beta^2-\rho_1\rho_2)}}{\beta^2-\rho_1\rho_2}
  \;\beta.
\end{align}
When $|s|$ is small, it can be readily seen
that $\vartheta_1$ is increasing and $\vartheta_2$ is decreasing
on the interval $(\sqrt{\rho_1\rho_2},\rho_1)$.
Also we have $\vartheta_2(\rho_1)>0$ since $w_{2L}>0$.

Let $\sigma(\tau)$ be the solution to \eqref{deq_fast_L}
with initial condition $\sigma(0)=(\rho_1,v_L)$.
By the monotonicity of $\vartheta_1$ and $\vartheta_2$,
we know that $\sigma$ hits the half-line $l_1= \{(\beta_L,v): v\ge v_L\}$
at some time $\tau_-<0$.
Let $l_2=\{(\rho_1,v): v\ge v_L\}$ and
$V$ be the region enclosed by the curves \[
  l_1\cup \{\sigma(\tau): \tau_-\le \tau\le 0\}\cup l_2.
\] Then $V$ forms a backward invariant region.
See Fig \ref{fig_h3_trapping}.

We claim that
$u_L$ attracts every point in $V$ in backward time.
Note that \[
  &\frac{\partial}{\partial\beta}(vB_1(\beta)-s\beta-w_{10})
  +\frac{\partial}{\partial v}(v^2B_2(\beta)-sv-w_{2L})\\
  &=vB_1'(\beta)-s+2vB_2(\beta)-s\\
  &=4vB_2(\beta)-2s,
\] which is positive when $\beta\in (\sqrt{\rho_1\rho_2},\rho_1)$
and $s$ is small.
In the last equality we used $B_1'(\beta)=2B_2(\beta)$.
By Bendixson's negative criterion, the system has no periodic orbit inside $V$.
Since $V$ is backward invariant and $u_L$ is the only equilibrium on the closure of $V$,
it follows from the Poincar\'e-Bendixson Theorem that
every trajectory in $V$ tends to $u_L$ in backward time.

It remains to show that there is a trajectory in $V$
tending to $\{\beta=\rho_1,v=\infty\}$.
Let $r=1/v$. Then \eqref{deq_fast_L} is converted to,
after multiplying by $r$, \beq{deq_fast_L_r}
  &\dot \beta= B_1(\beta)- s\beta r- w_{10}r\\
  &\dot r= -rB_2(\beta)+ sr^2+ w_{2L}r^3.
\] At the equilibrium $p_L=(\rho_1,0)$,
the eigenvalues of the linearized system are
$\lambda_+=1-\frac{\rho_2}{\rho_1}$
and $\lambda_-=\frac{-1}2(1-\frac{\rho_2}{\rho_1})$,
and the corresponding eigenvectors are $y_+=(1,0)$
and $y_-=(\frac{2\rho_1(w_{1L}+s\rho_1)}{3(\rho_1+\rho_2)},1)$,
so $p_L$ is a hyperbolic saddle,
and hence there exists a trajectory, denoted by $\gamma_1$,
which tends to $p_L$.
The trajectory of $\gamma_1$ is tangent to the line $\{p_L+ty_-:t\in\mathbb R\}$ at $p_L$.
Since $s\rho_1+w_{1R}<0$ by Lemma \ref{lem_bdry_sign},
we know $p_L+ty_1$, $t\ge 0$, lies in the region $\{\beta<\rho_1,r\ge 0\}$.
Therefore, converting \eqref{deq_fast_L_r} back to \eqref{deq_fast_L},
the solution converted from $\gamma_1(\tau)$,
also denoted by $\gamma_1(\tau)$, lies in $V$.
Now we conclude that
$\gamma_1(\tau)$ approaches $\{\beta=\rho_1,v=\infty\}$ in forward time
and approaches $u_L$ in backward time.
\end{proof}

\section*{Acknowledgments}
The author wishes to thank Professor Maciej Krupa
for his fruitful comments.
Also the author wishes to express his gratitude to
his advisor, Professor Barbara L.\ Keyfitz,
for her constant patience and encouragement,
and for the inspiration he received
from their many stimulating conversations,
all of which contributed to the completion of this work.


\end{document}